\documentclass{article}
\usepackage[utf8]{inputenc}

\usepackage{longtable}
\usepackage{latexsym,amsmath,amstext,amsfonts,amscd,amsthm,upref,graphics}
\usepackage{enumitem,lipsum}
\usepackage{graphicx}
\usepackage{color,xcolor}
\usepackage[]{algorithm2e}
\usepackage{float}
\restylefloat{table}
\usepackage{url}

\newcommand{\ceiling}[1]{ \left\lceil #1  \right\rceil}

\usepackage[utf8]{inputenc}

\newcommand{\ask}{$^*$}
\newcommand{\car}{$_{\hspace{.3em}\land}$}
\newcommand{\askcar}{$^*_{\hspace{.3em}\land}$}

\newcommand{\sccd }{SCCD}
\newcommand{\tsccd }{tight~SCCD}
\newcommand{\esccd }{economic~SCCD}
\newcommand{\scccd}{CSCCD}
\newcommand{\tscccd}{tight~CSCCD}
\newcommand{\escccd}{economic~CSCCD}

\newcommand{\cE}{\mathcal{E}}
\newcommand{\cL}{\mathcal{L}}
\newcommand{\cM}{\mathcal{M}}

\newcommand{\bN}{\mathbb{N}}

\definecolor{colour1}{HTML}{009901}
\definecolor{colour2}{HTML}{FE0000}
\definecolor{colour3}{HTML}{3531FF}
\definecolor{colour4}{HTML}{E64EE5}
\definecolor{colour5}{HTML}{F56B00}
\definecolor{colour6}{HTML}{6200C9}
\definecolor{colour7}{HTML}{AAAD3A}

\usepackage{tikz}
\usetikzlibrary{graphs,graphs.standard}

\usepackage[utf8]{inputenc}
\usepackage[nottoc]{tocbibind}

\theoremstyle{plain}
\newtheorem{theorem}{Theorem}[{}]

\newtheorem{lemma}[theorem]{Lemma}
\newtheorem{proposition}[theorem]{Proposition}
\newtheorem{corollary}[theorem]{Corollary}

\usepackage{setspace}
\begin{document}

\title{Linear and Circular Single Change Covering Designs Re-visited}
\author{Amanda Chafee,  Brett Stevens\\
School of Mathematics and Statistics\\
Carleton University\\
1125 Colonel By Drive\\
Ottawa, ON, K1S 5B6\\
Canada\\
\texttt{AmandaChafee@cmail.carleton.ca}, \texttt{brett@math.carleton.ca}
}

\date{\today}

\maketitle

\section{Abstract}
A \textbf{single change covering design} is a $v$-set $X$ and an ordered list $\cL$ of $b$ blocks of size $k$  where every $t$-set must occur in at least one block. Each pair of consecutive blocks differs by exactly one element. A single change covering design is circular when the first and last blocks also differ by one element. A single change covering design is minimum if no other smaller design can be constructed for a given $v, k$. 

In this paper we use a new recursive construction to solve the existence of circular \sccd($v,4,b$) for all $v$ and three residue classes of circular \sccd($v,5,b$) modulo 16. We solve the existence of three residue classes of \sccd$(v,5,b)$ modulo 16. We prove the existence of circular \sccd$(2c(k-1)+1,k,c^2(2k-2)+c)$, for all $c\geq 1, k\geq2 $, using difference methods. 

Key Words: Design theory, change designs, single change covering designs, difference families

\section{Introduction}

A \textbf{single change covering design} (\sccd($v,k,b$)) $(X,\mathcal{L})$ is a $v$-set $X$ and an ordered list $\mathcal{L}= (B_1,B_2,...,B_b)$ of blocks of size $k$  where every pair must occur in at least one block. Each pair of consecutive blocks differs by exactly one element. A \textbf{singe change covering design} (\scccd) is circular if the first and last blocks also differ by one element. These types of \sccd~ are considered strength two as every pair is present in at least one block. Instead of pairs we could require that each triple or larger $t$-sets occur in at least one block to build strength $t$ designs. 

The earliest discussion of \sccd s was in 1969 and focused on applications to efficient computing \cite{Nelder}. Consider the calculation of $AA^T$ for some matrix; this requires processing all pairs of rows of $A$. In 1969 a limited number of pairs could be stored in core memory (RAM) at one time and updating RAM from a tape was slow. Therefore, visiting all pairs with the smallest number of accesses to the tape improves running time significantly. Similar constraints apply today with respect to CPU cache vs. RAM access \cite{Tape_Drive, Core_Memory}.  

Gower and Preece continued the discussion in 1972 when they analyzed minimizing changes between successive blocks \cite{Gower}. Work continued in the 1990s by various subsets of Preece, Constable, Zhang, Yucas, Wallis, McSorley and Phillips in a series of eleven papers that substantially developed the theory. Most relevant to this work is the solution Preece, et al. give about the existence of minimum \sccd ~with $k=4$ using a recursive construction based solution on ``expansion sets'' \cite{tsccd}.

In 1998 McSorley explored circular single change covering designs \cite{circular}. He completely solved minimum \scccd($v,k,b$) where $k=2,3$ and where $v \leq 2k$. Analysing properties of elements in a design which are only introduced once, he explored substructures that must be present in circular designs to significantly reduce the work required in a computer search which lead to him constructing all the minimum \scccd(9,4,12) and \scccd(10,4,15). 

In this previous research, the powerful recursive construction used to successfully complete some families of linear \sccd ~had no analog for circular \sccd. We introduce such a recursive construction and completely solve the existence of circular \sccd ~for $k=4$ and for three residue classes of $v \pmod{16}$ when $k=5$. We also solve the existence for linear \sccd ~for three residue class of $v \pmod{16}$ when $k=5$. Further, we construct an infinite family of \scccd ~for every fixed $k$ using difference methods.


\section{Background}

In a \sccd ~every block except the last  differs from the next block by a single change. That is $|B_i \cap B_{i+1}| = k-1$ for all $1 \leq i < b$. We say the element $x=B_{i} \backslash B_{i+1}$ is \textbf{{removed}} from $B_i$ and $y=B_{i+1} \backslash B_{i}$ is \textbf{{introduced}} in $B_{i+1}$. If the \sccd ~is circular, one element is introduced in each block. If the \sccd ~is linear, every element in $B_1$ is introduced and one element is introduced in each subsequent block. A pair is \textbf{{covered}} on block $B_i$ if $S \subseteq B_i$ and one element of $S$ was introduced in $B_i$. In this paper we will only consider pairs, but the study of \sccd ~for higher strength sets is of interest. We say a \sccd$(v,k,b)$ is \textbf{minimum} if no \sccd$(v,k,b')$ exists for $b'<b$.

An example \sccd(7,3,10) is given in Table~\ref{Table: tsccd(7,3,2)}. 
We emphasize the introduced element in each block with an asterisk, \textbf{*}. Note that elements 1, 2, and 3 are all introduced in the first block so the pairs $\{1,2\},\{1,3\},\{2,3\}$ are covered in $B_1$. In $B_2$ only 4 is introduced, so only pairs involving 4 are covered, namely $\{1,4\},\{2,4\}$. 

\begin{table}[H]
\centering
\begin{tabular}{llllllllll}
$B_1$ & $B_2$ & $B_3$ & $B_4$ & $B_5$ & $B_6$ & $B_7$ & $B_8$ & $B_9$ & $B_{10}$\\
1* & 1  & 1  & 6* & 7* & 7  & 4* & 4  & 4   & 1* \\
2* & 2  & 2  & 2  & 2  & 3* & 3  & 3  & 7*  & 7  \\
3* & 4* & 5* & 5  & 5  & 5  & 5  & 6* & 6   & 6
\end{tabular}
\caption{\sccd(7,3,10)}
\label{Table: tsccd(7,3,2)}\label{Table: sccd(7,3,2)} \label{Table: sccd(7,3,2,10)}
\end{table}

In Table~\ref{Table: tsccd(7,3,2)} we note that every pair is covered exactly once. However, in Table~\ref{Table: escccd(6,3,2)} we have a \scccd(6,3,8) where the pair $\{1,4\}$ is covered two times.

\begin{table}[H]
\centering
\caption{\scccd(6,3,8). The pair \{1,4\} is covered in $B_5$ and $B_8$. \cite{circular}}
\label{Table: escccd(6,3,2)} \label{Table: scccd(6,3,2,8)}
\begin{tabular}{ll l lll l l}
$B_1$ & $B_2$ & $B_3$ & $B_4$ & $B_5$ & $B_6$ & $B_7$ & $B_8$    \\
6\ask & 6     & 6        & 6     & 4\ask & 4     & 2\ask & 2     \\
4     & 3\ask & 3        & 3     & 3     & 5\ask & 5     & 4\ask \\
2     & 2     & 5\askcar & $1^*$ & 1     & 1\car & 1     & 1\car
\end{tabular}
\end{table}

The set of elements that remain the same between $B_i$ and $B_{i+1}$ of a \sccd ~is the $i^{th}$ \textbf{unchanged subset} between these blocks, $U_i=B_i \cap B_{i+1}$ where $i+1$ is taken modulo $b$ in a \scccd \cite{tsccd}. If the \sccd ~is not circular, $U_0$ and $U_b$  can be any $(k-1)$-subset of $B_1$ or $B_b$ respectively. In the \sccd(7,3,10) in Table~\ref{Table: tsccd(7,3,2)}, the unchanged subsets are $U_1$=$\{1,2\},$ $U_2$=$\{1,2\},$ $U_3$=$\{2,5\},$ $U_4$=$\{2,5\},$ $U_5$=$\{7,5\},$ $U_6$=$\{3,5\},$ $U_7$=$\{4,3\},$ $U_8=\{4,$ $6\},$ $U_9$=$\{7,6\}$. Additionally, $U_0$ could be $\{1,2\},$ $\{1,3\}$, or $\{2,3\}$, and $U_{10}$ could be $\{1,7\},$ $\{6,7\},$ or $\{1,6\}$. 
If there exists $\{i_j: 1\leq j\leq \frac{v}{k-1}\}$ a disjoint union of unchanged subsets such that 
$$ X = \dot\bigcup^{\frac{v}{k-1}}_{j=1} U_{i_j} $$ 
then we say that $\cE = \{U_{i_1},U_{i_2},...,U_{i_{\frac{v}{(k-1)}} }\}$ is an \textbf{{expansion set}}. If the \sccd ~is not circular and $\cE$ contains $U_0$, $U_b$ or both, then $\cE$ is an \textbf{{outer expansion set}}, otherwise $\cE$ is an \textbf{{inner expansion set}}. We will denote expansion set locations in our tables with carets, \textbf{$\land$}. In Table~\ref{Table: escccd(6,3,2)}, $ \cE = \{U_3=\{3,6\}, U_6=\{1,5\}, U_8=\{2,4\} \}$ is an expansion set.

Let $g_1(v,k,2) = \frac{\binom{v}{2}-\binom{k}{2}}{k-1}+1$ and $g_2(v,k,2) = \frac{\binom{v}{2}}{k-1}$. Wallis et al. showed.
 
\begin{lemma}~\label{Lemma: Number of blocks sccd(v,k,2)} \cite{sccd}
In a \sccd($v,k,b$), $b \geq g_1(v,k,2).$ In a circular \newline\sccd($v,k,b$), $b\geq g_2(v,k,2).$
\end{lemma}

We say a (circular) \sccd($v,k,b$) is \textbf{{economical}} if it has $\ceiling{g_1(v,k,2)}$ $\big(\ceiling{g_2(v,k,2)}\big)$ blocks. A (circular) \sccd($v,k,b$) is \textbf{{tight}} if it is economical and $g_1(v,k,2)$ $\big(g_2(v,k,2)\big)$ is an integer. Any economical \sccd ~is minimum. In a tight \sccd ~every pair is covered exactly once. For some $v,k$ combinations tight \sccd ~can not exist, but economical designs can exist. 

Although not formally stated, Preece et al. proposed two recursive constructions for \tsccd \cite{tsccd}. We state them explicitly.

\begin{proposition}[$v+1$ construction]\cite{tsccd} \label{Prop: circular sccd(v+1,k)}
If a \sccd$(v,k,b)$ with an expansion set exists then a $\sccd(v+1,k, b+\frac{v}{k-1} )$ exists. If the \sccd$(v,k,b)$ is tight then so is the $\sccd(v+1,k, b+\frac{v}{k-1} )$. 
\end{proposition}

For example, consider the \sccd(13,4,25) in Table~\ref{Table: sccd(13,4,2,25)} with the expansion set $\cE = \{ U_1=\{1,2,3\}, U_7=\{5,8,9\}, U_{13}=\{4,6,7\}, U_{22}=\{10,11,12\} \}$. Preece et al. construct the \sccd(13,4,25) given in Table~\ref{Table: sccd(13,4,2,25)} by introducing a new block containing the new element with each unchanged subset in the expansion set of a \sccd(12,4,21).

\begin{table}[H]
\centering
\caption{\sccd(13,4,25)}
\label{Table: sccd(13,4,2,25)}\label{Table: sccd(13,4,25)}
\setlength{\tabcolsep}{.55\tabcolsep}
\begin{tabular}{lllllllllllllll}
$B_1$ & $B_2$ & $B_3$ & $B_4$ & $B_5$ & $B_6$ & $B_7$ & $B_8$ & $B_9$ & $B_{10}$ & $B_{11}$ & $B_{12}$ & $B_{13}$ & $B_{14}$ & $B_{15}$    \\

1$^*$  & \textcolor{blue}{1} & 1  & 1  & 1  & 1  & 1  & \textcolor{blue}{13}$^*$  & 4$^*$  & 4  & 4  & 4  & 4  & \textcolor{blue}{4}  & 4  \\

2$^*$  & \textcolor{blue}{2} & 2  & 2  & 2  & 2  & 9$^*$  & \textcolor{blue}{9}  & 9  & 9  & 10$^*$ & 11$^*$ & 11 & \textcolor{blue}{13}$^*$ & 12$^*$ \\

3$^*$  & \textcolor{blue}{3} & 3  & 6$^*$  & 7$^*$  & 8$^*$  & 8  & \textcolor{blue}{8}  & 8  & 8  & 8  & 8  & 7$^*$  & \textcolor{blue}{7}  & 7  \\

4$^*$  & \textcolor{blue}{13}$^*$ & 5$^*$  & 5  & 5  & 5  & 5  & \textcolor{blue}{5}  & 5  & 6$^*$  & 6  & 6  & 6  & \textcolor{blue}{6}  & 6  \\ \hline

$B_{16}$ & $B_{17}$ & $B_{18}$ & $B_{19}$ & $B_{20}$ & $B_{21}$ & $B_{22}$ & $B_{23}$ & $B_{24}$ & $B_{25}$   \\

3$^*$  & 3                        & 3  & 3  & 3  & 2$^*$  & 2  & \textcolor{blue}{13}$^*$  & 5$^*$  & 1$^*$  &    &    &    &                           &    \\

12 & 12                       & 12 & 10$^*$ & 10 & 10 & 10 & \textcolor{blue}{10} & 10 & 10 &    &    &    &                           &    \\

7  & 7                        & 7  & 7  & 11$^*$ & 11 & 11 & \textcolor{blue}{11} & 11 & 11 &    &    &    &                           &    \\

6  & 8$^*$                        & 9$^*$  & 9  & 9  & 9  & 12$^*$ & \textcolor{blue}{12} & 12 & 12 &    &    &    &                           &   
\end{tabular}
\end{table}

\begin{theorem}[Building \tsccd ~from two \tsccd s] \label{Theorem: tsccd} \label{Theorem: tsccd +tsccd = tsccd} \cite{tsccd}
If there exists a \tsccd$(v,k,b)$ and a tight \sccd$(v',k,b')$ with an outer expansion set where $|X \cap X'|=k-1$, then for $v^*=v+v'-k+1$,  $b^*= b + b' + \frac{(v-k+1)(v'-k+1)}{k-1}$ a \tsccd$(v^*,k,b^*)$ exists where $X^*=X\cup X'$. Furthermore, if the \tsccd$(v,k,b)$ has an expansion set then the \tsccd$(v^*,k,b^*)$ has an expansion set.
\end{theorem}

For example, Table~\ref{tsccd(6,3) D} shows two tight \sccd$(6,3,7)$. The \tsccd(10,3,22), constructed by Theorem~\ref{Theorem: tsccd +tsccd = tsccd} is shown in Table~\ref{tsccd(10,3)} with the blocks of $\cL$, $\cL'$, and $B''_{i_j,x}$ shown in  black, red, and blue respectively.

\begin{table}[H]
\setlength{\tabcolsep}{.55\tabcolsep}
\centering
\caption{Two \tsccd(6,3,7) with outer expansion sets}
\label{tsccd(6,3) D} \label{tsccd(6,3,2,7) D}
\label{tsccd(6,3) D'} \label{tsccd(6,3,2,7) D'}
\begin{tabular}{p{15pt} p{10pt} lll p{10pt} l       lll p{15pt} p{10pt} lll p{10pt} l}
$\hspace{0.6em}$B$_1$ & B$_2$ & B$_3$ & B$_4$ & B$_5$ & B$_6$ & B$_7$ 
&&&& $\hspace{0.6em}$B'$_1$ & B'$_2$ & B'$_3$ & B'$_4$ & B'$_5$ & B'$_6$ & B'$_7$ \\
$\hspace{0.6em}1^*$ & 1 & 1 & 1 & $3^*$ & 3 & $2^*$ 
&&&&$\hspace{0.6em}\textcolor{red}{2^*}$ & \textcolor{red}{2} & $\textcolor{red}{9^*}$ & $\textcolor{red}{10^*}$ & \textcolor{red}{10} & \textcolor{red}{10} & \textcolor{red}{10} \\
$\hspace{0.6em}2^*$ & 2 & $5^*$ & $6^*$ & 6 & 6 & 6  
&&&& $\hspace{0.6em}\textcolor{red}{7^*}$ & \textcolor{red}{7} & \textcolor{red}{7} & \textcolor{red}{7} & $ \textcolor{red}{5^*}$ & \textcolor{red}{5} & $\textcolor{red}{2^*}$\\
$_{\land}3^*$ & $4^*_{\hspace{.3em}\land}$ & 4         & 4         & 4         & $5^*_{\hspace{.3em}\land}$ & 5  
&&&& $\textcolor{red}{_{\land}5^*}$ & $\textcolor{red}{8^*_{\hspace{.3em}\land} }$ & \textcolor{red}{8}         & \textcolor{red}{8}         & \textcolor{red}{8}         & $\textcolor{red}{9^*_{\hspace{.3em}\land} }$ & \textcolor{red}{9} 
\end{tabular}
\end{table}

\begin{table}[H]
\centering
\setlength{\tabcolsep}{.55\tabcolsep}
\caption{A \tsccd(10,3,22) with outer expansion set}
\label{tsccd(10,3)} \label{tsccd(10,3,2,22)}
\begin{tabular}{p{13pt} p{8pt} p{8pt}p{8pt}p{8pt} p{8pt} p{8pt}p{8pt}    p{8pt} p{8pt} p{8pt}p{8pt}p{8pt} p{8pt}p{10pt}p{10pt}p{10pt} p{10pt} p{10pt}p{10pt}p{10pt}p{10pt}p{10pt} }
$\hspace{0.6em}$\small{B}\tiny{$_1$} & \small{B}\tiny{$_2$} & \small{B}\tiny{$_3$} & \small{B}\tiny{$_4$} & \small{B}\tiny{$_5$} & \small{B}\tiny{$_6$} & \small{B}\tiny{$_7$} & \small{B}\tiny{$_8$} & \small{B}\tiny{$_9$} & \small{B}\tiny{$_{10}$} & \small{B}\tiny{$_{11}$} & \small{B}\tiny{$_{12}$} & \small{B}\tiny{$_{13}$} & \small{B}\tiny{$_{14}$} & \small{B}\tiny{$_{15}$} & \small{B}\tiny{$_{16}$} & \small{B}\tiny{$_{17}$} & \small{B}\tiny{$_{18}$} & \small{B}\tiny{$_{19}$} & \small{B}\tiny{$_{20}$} & \small{B}\tiny{$_{21}$} & \small{B}\tiny{$_{22}$} \\
$\hspace{0.6em}1^*$ & 1 & 1     & 1     & $3^*$ & 3 & $2^*$ & $\textcolor{red}{2}$ & \textcolor{red}{2} & $\textcolor{blue}{1^*}$ & $\textcolor{blue}{3^*}$ & $\textcolor{blue}{4^*}$ & $\textcolor{blue}{6^*}$ &$\textcolor{red}{9^*}$ & $\textcolor{red}{10^*}$ & \textcolor{red}{10}    & \textcolor{red}{10} & \textcolor{blue}{10} & \textcolor{blue}{10} & \textcolor{blue}{10} & \textcolor{blue}{10} & \textcolor{red}{10} \\

$\hspace{0.6em}2^*$ & 2 & $5^*$ & $6^*$ & 6     & 6 & 6     & $\textcolor{red}{7^*}$   & \textcolor{red}{7} & \textcolor{blue}{7} & \textcolor{blue}{7} & \textcolor{blue}{7} & \textcolor{blue}{7} & \textcolor{red}{7}     & \textcolor{red}{7}      & $\textcolor{red}{5^*}$ & \textcolor{red}{5}  & $\textcolor{blue}{1^*}$ & $\textcolor{blue}{3^*}$ & $\textcolor{blue}{4^*}$ & $\textcolor{blue}{6^*}$ & $\textcolor{red}{2^*}$\\

$_{\land}3^*$ & $4^*_{\hspace{.3em}\land}$ & 4 & 4 & 4 & $5^*_{\hspace{.3em}\land}$ & 5   & $\textcolor{red}{5}$ & $\textcolor{red}{8^*_{\hspace{.3em}\land} }$ & \textcolor{blue}{8} & \textcolor{blue}{8} & \textcolor{blue}{8} & \textcolor{blue}{8} & \textcolor{red}{8} & \textcolor{red}{8} & \textcolor{red}{8} & $\textcolor{red}{9^*_{\hspace{.3em}\land} }$ & \textcolor{blue}{9} & \textcolor{blue}{9} & \textcolor{blue}{9} & \textcolor{blue}{9} & \textcolor{red}{9}
\end{tabular}
\end{table}

As of 2001, the following \sccd ~and \scccd ~were known to exist. 
\begin{theorem}\label{Theorem: Existance of sccd (v,2,2,b),(v,3,2,b),(v,4,2,b),(v,5,2,b)} \label{Theorem: Previously known groups of sccd}
    \begin{enumerate}
        \item There exists a tight \sccd$(v,2,b)$ for all $v$
        \cite{Nelder}.
        \item An economical \sccd$(v,3,b)$ exists for all $v \geq 6$
        ; tight if and only if $v\equiv 2,3\pmod{4}$\cite{sccd}.
        \item An economical \sccd$(v,4,b)$ exists for all $v \geq 12$
        ; tight if and only if $v\equiv 0,1\pmod{3}$ \cite{Zhang_new_bounds,tsccd}.  \label{Theorem: Existance tsccd(v,4,2)}
        \item A \tsccd$(20,5,46)$ exists \cite{Phillips_20}.
        \item A tight \scccd$(v,2,b)$ exists for all $v \geq 3$
        \cite{circular}.
        \item An economical \sccd($v,3,b$) exists for all $v \geq 4$, $b=\ceiling{\frac{v(v-1)}{4}}$; tight if and only if $v \equiv 0,1 \pmod{4}$ \cite{circular}.
        \item A \tscccd(9,4,12) and \tscccd(10,4,15) exist \cite{circular}.
        \item An economical \scccd$(v,k,v-1)$ exists for $k+1 \leq v \leq 2k-2$ \cite{circular}.
        \item A \tscccd$(2k-2,k,2k-3)$ exists for $k\geq 3$\cite{circular}.
        \item A \tscccd$(2k-1,k,2k-1)$ exists for $k\geq 2$ \cite{circular}.
        \item An \escccd$(2k,k,2k+2)$ exists for $k\geq 2$ \cite{circular}.
    \end{enumerate}
\end{theorem}

In this paper, we extend Theorem~\ref{Theorem: tsccd +tsccd = tsccd} to include economical \sccd ~and circular \sccd. We use this to almost complete the spectrum of \scccd ~with $k=4$ and three residue classes of \sccd ~and \scccd ~with $k=5$. There are still some small open cases when $k=4$ and $k=5$. Furthermore, we provide an infinite number of \scccd ~for each $k$ using difference methods.

\section{Results}

We first will generalize Theorem~\ref{Theorem: tsccd} to construct economical \sccd. We say that a block $B_i$ in a \sccd ~is \textbf{tight} if the pairs it covers are not covered in any other block of the \sccd. We say the \textbf{excess}, $e$, of a \sccd ~is the number of pairs covered repeatedly. 
\begin{equation*}
e=\begin{cases}
          (k-1)b+\binom{k-1}{2}-\binom{v}{2} \quad & \text{in linear \sccd}\\
          (k-1)b-\binom{v}{2} \quad & \text{in Circular \sccd}\\
     \end{cases}
\end{equation*}

We say the \textbf{excess}, $e_i$, of a block $B_i$ in a \sccd ~is the number of pairs $B_i$ covers that were covered in any block $B_j$, $j<i$. We say the \textbf{excess}, $e_i$, of a block $B_i$ with respect to an initial block $B_0$ in a \scccd ~is the number of pairs $B_i$ covers that were covered in any block $B_j$, $0\leq j < i$.

\begin{lemma}\label{Lemma: The Excess of a sccd is the sum of excesses of blocks}
The excess, $e$, of a (circular) \sccd ~is the sum of the excesses of blocks.
\end{lemma}

\begin{proof}
Consider the \sccd$(v,k,b)$, $(X,\cL)$. Each block $B_i$ will cover $k-1$ pairs where $e_i$ of these have been previously covered. Summing through these blocks we have that $\sum_{i=1}^b e_i = e$. The proof for the circular \sccd ~is similar with respect to an initial block.
\end{proof}

\begin{lemma}\label{Lemma: excess}
A (circular) \sccd ~is economical if and only if $e\leq k-2$. A \sccd ~is tight if and only if $e=0$.
\end{lemma}
\begin{proof}
Suppose that we have an economical \sccd($v,k,b$). 
From the definition of $e$ we find
\begin{align*}
    e & = (k-1)b+\binom{k-1}{2}-\binom{v}{2}\\ 
      & = \ceiling{\frac{\binom{v}{2}-\binom{k}{2}}{k-1}+1}(k-1)-\binom{v}{2}+\binom{k-1}{2} \\
      & \text{Since } \frac{n}{m} \leq \ceiling{\frac{n}{m}}<\frac{n}{m}+1 \text{ we have that}\\
    e & < \binom{v}{2}-\binom{k}{2}+2(k-1) - \binom{v}{2} + \binom{k-1}{2}\\
    e & < k-1
\end{align*}

Conversely, suppose $0 \leq e <k-1$, since $e \geq 0$ we have $\frac{\binom{v}{2}-\binom{k}{2}}{k-1}+1 \leq b$ and $b=\frac{\binom{v}{2}-\binom{k}{2}}{k-1}+\frac{e}{k-1}+1 <\frac{\binom{v}{2}-\binom{k}{2}}{k-1}+1+1$. Thus since $b\in \bN$, $b=\ceiling{\frac{\binom{v}{2}-\binom{k}{2}}{k-1}+1}$.

A \sccd$(v,k,b)$ is tight if every pair is covered exactly once, therefore no pair can be previously covered and $e=0$. 

The proof for circular \sccd ~is similar.
\end{proof}

Noting that the excess of a tight block is zero proves the following. 

\begin{lemma}\label{Lemma: esccd with tight blocks is esccd}
Let $(X,\mathcal{L}=(B_1,...,B_b))$ be a $\sccd(v,k,b)$ and let  $(X',\mathcal{L}'=(B_1,...,B_b,B'_{b+1},...,B'_{b'}))$ be a $\sccd(v',k,b')$ with $X \subseteq X'$. If $B'_i$ is tight $\forall i, b+1 \leq i \leq b'$, then $(X',\mathcal{L}')$ has the same excess as $(X,\cL)$. 
\end{lemma}


For example, consider the \esccd$(8,3,14)$ of Table~\ref{Table: esccd(8,3,2,14) built with tight blocks} whose first three blocks, given in red, are an economic \sccd($4,3,3$) with excess 1. Blocks $4-14$ are tight so the \sccd($8,3,14$) has excess 1.

\begin{table}[H]
\centering
\caption{\esccd(8,3,14)}
\label{Table: esccd(8,3,2,14) built with tight blocks}\label{Table: esccd(8,3,14) built with tight blocks}

\begin{tabular}{llllllllllllll}
\small{B}\tiny{$_1$} & \small{B}\tiny{$_2$} & \small{B}\tiny{$_3$} & \small{B}\tiny{$_4$} & \small{B}\tiny{$_5$} & \small{B}\tiny{$_6$} & \small{B}\tiny{$_7$} & \small{B}\tiny{$_8$} & \small{B}\tiny{$_9$} & \small{B}\tiny{$_{10}$} & \small{B}\tiny{$_{11}$} & \small{B}\tiny{$_{12}$} & \small{B}\tiny{$_{13}$} & \small{B}\tiny{$_{14}$} \\
\hspace{6pt}\textcolor{red}{$1^*$} & \textcolor{red}{$4^*$} & \textcolor{red}{4}     & 4     & $7^*$ & 7     & 7     & 7     & 7     & 7     & $6^*$ & 6     & 6     &6  \\
\hspace{6pt}\textcolor{red}{$2^*$} & \textcolor{red}{2}     & \textcolor{red}{$1^*$}     & $8^*$ & 8     & 8     & 8     & 8     & 8     & $4^*$ & 4     & $1^*$ & $2^*$ &$3^*$ \\
$_{\land}$\textcolor{red}{$3^*$} & \textcolor{red}{3}     & \textcolor{red}{3}$_{\land}$ & 3     & 3     & $6^*_{\land}$ & $1^*$ & $2^*$ & $5^*$ & 5     & 5     & 5     & 5     &5$_{\land}$
\end{tabular}
\end{table}

We note that we can reverse the design Lemma~\ref{Lemma: esccd with tight blocks is esccd} produced and generate a design where the tight blocks are at the start of the design. 

We may now generalize Theorem~\ref{Theorem: tsccd +tsccd = tsccd}.

\begin{theorem}[Constructing \sccd ~with excess using two \sccd]\label{Theorem: esccd + tsccd = esccd}
If there exists a \sccd(v,k,b) with excess $e$ and a \sccd$(v',k,b')$ with an outer expansion set and excess $e'$, then a \sccd$(v^*,k,b^*)$, $v*=v+v'-k+1$, $b^* = b +b' + \frac{(v-k+1)(v'-k+1)}{k-1}$, exists with excess $e^*=e+e'$. Furthermore, if \sccd$(v,k,b)$ has an expansion set then \sccd$(v^*,k,b^*)$ has an expansion set.
\end{theorem}

\begin{proof} Suppose $(X,\mathcal{L})$ is a \sccd$(v,k,b)$ with excess $e$ and $(X',\mathcal{L}')$ a \newline \sccd$(v',k,b')$ with excess $e'$ and an outer expansion set $\mathcal{E}' = \{U'_{i_j}: 1 \leq j \leq \frac{v'}{k-1}\}$. Since the reverse of a \sccd ~is a \sccd ~we may assume $\cE'$ contains $U'_0$. We re-label the elements of $(X',\cL')$ so that $B_{b} \cap B'_1 = X \cap X' =U'_0$ and $X, X'$ are otherwise disjoint. To build $(X^*,\cL^*)$ we start by appending $\cL$ with $\cL'$.

Since $(X',\cL')$ has an outer expansion set and the first expansion location of $(X',\mathcal{L}')$ is $U'_{0} = B_b\cap B'_1$ we have that the remaining expansion locations partition $X' \backslash X$. For all $x\in X \backslash X'$ and $2\le j\le \frac{v'}{k-1}$, we let $B''_{i_j,x} = U'_{i_j} \cup \{x\}$ and insert $B''_{i_j,x}$ between $B'_i$ and $B'_{i+1}$ in any order; $\cL''$ is single change. 

The blocks $\cL$ and $\cL'$ cover the same pairs in $\cL^*$ as they did in $\cL$ and $\cL'$, with the exception of $B'_1$. In $\cL'$, $B'_1$ covered all of its pairs. However, in $\cL^*$, $B'_1$ only covers the pairs $\{x,y\}$ where $x=B'_1 \backslash B_b$, and $y\in B'_1 \cap B_b$; every other pair in $B'_1$ is covered in $\cL$. The only pairs not covered in $\cL$ or $\cL'$ are $\{x,y\}$ where $x \in X \backslash X'$ and $y \in X' \backslash X=\bigcup\limits_{j=2}^{n} U'_{i_j}$. $B''_{i_{j},x}$ covers $\{x,y\}$ for any $x \in X \backslash X'$ and  all $y \in U_{i_j}$ and no other $B''_{i_{j},x}$ covers this pair. Therefore $(X^*,\mathcal{L}^*)$ is a \sccd. 

\begin{align*}
    e^* =& (k-1)b^* + \binom{k-1}{2}-\binom{v^*}{2} \\
        =& (k-1)(b + b'+\frac{(v-k+1)(v'-k+1)}{k-1}) + \binom{k-1}{2} \\ 
         & - \binom{v+v'-k+1}{2}\\
        =& (k-1)b + (k-1)b' + vv'-v(k-1) -v'(k-1) +(k-1)^2 + \binom{k-1}{2} \\ 
         & - \frac{(v+v'-(k-1))(v+v'-(k-1)-1)}{2}\\
        =& (k-1)b + (k-1)b' + vv'-v(k-1) -v'(k-1) +(k-1)^2 + \binom{k-1}{2} \\ 
         & - \frac{v^2}{2} -vv' + v(k-1) + \frac{v}{2} - \frac{v'^2}{2} + v'(k-1) + \frac{v'}{2} - \frac{(k-1)^2}{2} - \frac{k-1}{2} \\
        =& (k-1)b + \binom{k-1}{2} - \frac{v^2}{2} + \frac{v}{2} + (k-1)b' + \frac{k-k^2 + 2(k-1)^2}{2} - \frac{v'^2}{2} + \frac{v'}{2} \\
        =& (k-1)b + \binom{k-1}{2} - \binom{v}{2} + (k-1)b' + \binom{k-1}{2} - \binom{v'}{2}\\
        =& e + e'
\end{align*}

Suppose now, that $(X,\mathcal{L})$, has an expansion set $\mathcal{E} = \{U_{i_j}: 1 \leq j \leq \frac{v}{k-1}\}$.  Then $(X^*,\mathcal{L}^*)$ will have an expansion set $\mathcal{E}^* = \{\mathcal{E} \cup \mathcal{E}' \backslash U'_{i_0} \}$.

\end{proof}

\begin{theorem}[$v+2$ Construction] \label{Theorem: create sccd(v+2,k)}\label{Theorem: sccd(v) -> sccd(v+2)}
If a \sccd$(v,k,b)$ with excess $e$ and an outer expansion set exists then a \sccd$(v+2,k,b')$, $b' = b+2\frac{v}{k-1}+1 $, exists with excess $e+k-2$.
\end{theorem}

\begin{proof} 
Let $(X,\mathcal{L}=(B_1,...,B_b))$ be a \sccd$(v,k,b)$ with an expansion set $\cE=\{U_{i_1},...,U_{i_{\frac{v}{k-1}}}=U_b\}$. Let $X'=X\dot \cup \{y_1,y_2\}$ build $(X',\cL')$ as follows. At each expansion location $U_i \in \mathcal{E}$ 
we will insert two blocks, $B'_{j_1}= U_{i_j} \cup \{y_1\}$ and $B'_{j_2}= U_{i_j} \cup \{y_2\}$, between $B_i$ and $B_{i+1}$. Let $x \in U_{\frac{v}{k-1}}$, after $B_{\frac{v}{k-1}_2}$ we will add the block $B'_{\frac{v}{k-1}_3} = U_b \backslash \{x\} \cup \{y_1, y_2\}$. The proof that this is a \sccd ~is similar to Proposition~\ref{Prop: circular sccd(v+1,k)} and Lemma~\ref{Lemma: esccd with tight blocks is esccd}. The excess of $B'_{j_1}$ and $B'_{j_2}$ are zero and the excess of blocks from $\cL$ remain unchanged. The excess of $B'_{\frac{v}{k-1}_3}$ is $k-2$, so 

\[
    e + \sum\limits_{j=1}^{\frac{v}{k-1}}(e'_{j_1} + e'_{j_2}) + e'_{\frac{v}{k-1}_3} = e + k-2
\]

\end{proof}


\begin{corollary}\label{Corollary: create esccd(v+2,k)}\label{Corollary: tsccd(v) -> esccd(v+2)}
    If there exists a \tsccd$(v,k,b)$ with an outer expansion set 
    then there exists an \esccd$(v+2,k,b')$.
\end{corollary}

For example, see Table~\ref{Table: esccd(14,4,2,30)}. The black blocks form a \tsccd(12,4,21). We insert the blue blocks to form the \esccd(14,4,30). 

\begin{table}[H]
\setlength{\tabcolsep}{.55\tabcolsep}
\centering
\caption{An \esccd(14,4,30)}
\label{Table: esccd(14,4,2,30)}
\footnotesize
    \begin{tabular}{lllllllllllllll}
        \small{B}\tiny{$_{1}$} & \small{B}\tiny{$_{2}$} & \small{B}\tiny{$_{3}$} & \small{B}\tiny{$_{4}$} & \small{B}\tiny{$_{5}$} & \small{B}\tiny{$_{6}$} & \small{B}\tiny{$_{7}$} & \small{B}\tiny{$_{8}$} & \small{B}\tiny{$_{9}$} & \small{B}\tiny{$_{10}$} & \small{B}\tiny{$_{11}$} & \small{B}\tiny{$_{12}$} & \small{B}\tiny{$_{13}$} & \small{B}\tiny{$_{14}$} & \small{B}\tiny{$_{15}$}   \\ 
        
        
        
        
        
        6$^*$  & 6  & \textcolor{blue}{6}  & \textcolor{blue}{6}  & 6  & 8$^*$  & 8  & 8  & 8  & \textcolor{blue}{8}  & \textcolor{blue}{8}  & 8 & 8  & 2$^*$  & 4$^*$  \\
        
        3$^*$  & 10$^*$ & \textcolor{blue}{10} & \textcolor{blue}{10} & 10 & 10 & 10 & 10 & 10 & \textcolor{blue}{13}$^*$ & \textcolor{blue}{14}$^*$ & 9$^*$ & 9  & 9  & 9  \\
        
        11$^*$ & 11 & \textcolor{blue}{11} & \textcolor{blue}{11} & 11 & 11 & 11 & 11 & 3$^*$  & \textcolor{blue}{3}  & \textcolor{blue}{3}  & 3 & 12$^*$ & 12 & 12 \\
        
        12$^*$ & 12$_{\land}$ & \textcolor{blue}{13}$^*$ & \textcolor{blue}{14}$^*$ & 4$^*$  & 4  & 2$^*$  & 7$^*$  & 7$_{\land}$  & \textcolor{blue}{7}  & \textcolor{blue}{7}  & 7 & 7  & 7  & 7  \\ \hline

        \small{B}\tiny{$_{16}$} & \small{B}\tiny{$_{17}$}   &
        \small{B}\tiny{$_{18}$} & \small{B}\tiny{$_{19}$} & \small{B}\tiny{$_{20}$} & \small{B}\tiny{$_{21}$} & \small{B}\tiny{$_{22}$} & \small{B}\tiny{$_{23}$} & \small{B}\tiny{$_{24}$} & \small{B}\tiny{$_{25}$} & \small{B}\tiny{$_{26}$} & \small{B}\tiny{$_{27}$} & \small{B}\tiny{$_{28}$} & \small{B}\tiny{$_{29}$} & \small{B}\tiny{$_{30}$}\\ 
        
        
        
        
        
        4  & \textcolor{blue}{13}$^*$ & \textcolor{blue}{14}$^*$ & 1$^*$  & 1  & 1  & 1  & 1  & 1  & 1  & 1  & 1 & \textcolor{blue}{1}  & \textcolor{blue}{1}  & \textcolor{blue}{1}  \\
        
        9  & \textcolor{blue}{9}  & \textcolor{blue}{9}  & 9  & 9  & 9  & 9  & 8$^*$  & 7$^*$  & 2$^*$  & 2  & 2 & \textcolor{blue}{2}  & \textcolor{blue}{2}  & \textcolor{blue}{2}  \\
        
        12 & \textcolor{blue}{12} & \textcolor{blue}{12} & 12 & 11$^*$ & 10$^*$ & 6$^*$  & 6  & 6  & 6  & 3$^*$  & 3 & \textcolor{blue}{13}$^*$ & \textcolor{blue}{14}$^*$ & \textcolor{blue}{14} \\
        
        5$^*$$_{\land}$  & \textcolor{blue}{5}  & \textcolor{blue}{5}  & 5  & 5  & 5  & 5  & 5  & 5  & 5  & 5  & 4$^*$$_{\land}$ & \textcolor{blue}{4}  & \textcolor{blue}{4}  & \textcolor{blue}{13}$^*$
        
    \end{tabular}
\end{table}

Now we consider constructions of circular \sccd ~using two linear \sccd s. We say an expansion set is \textbf{disjoint-capable} if 
\begin{enumerate}
    \item It contains $U_0$ and $U_b$
    \item $U_0=U_1=U_2=...=U_{k-1}$
    \item $U_b = \bigcup\limits^{k-1}_{i=1}(B_i \backslash B_{i+1})$. 
\end{enumerate}
 
The tight \sccd(10,3,22) in Table~\ref{Table: tsccd(10,3,2) extra property} of \cite{tsccd} contains a disjoint-capable expansion set:  $U_0=U_1=U_2= \{a,b\}$, and $U_b = \{c,d\}=(B_1\backslash B_2) \cup (B_2\backslash B_3)$.

\begin{table}[H]
\centering
\caption{Tight \sccd(10,3,22) with a disjoint-capable expansion set.}
\label{Table: tsccd(10,3,2) extra property}
\begin{tabular}{lllllllllll}

B$_1$ & B$_2$ & B$_3$ & B$_4$ & B$_5$ & B$_6$ & B$_7$ & B$_8$ & B$_9$ & B$_{10}$ & B$_{11}$ \\

$\hspace{0.6em}$\textcolor{orange}{a}$^*$ & \textcolor{orange}{a} & \textcolor{orange}{a} & a & a & 4$^*$ & 5$^*$ & 6$^*$ & 6 & 6 & 6 \\ 

$\hspace{0.6em}$\textcolor{orange}{b}$^*$ & \textcolor{orange}{b} & \textcolor{orange}{b} & 2$^*$ & 3$^*$ & 3 & 3 & 3  & c$^*$  & d$^*$  & d  \\ 

$_{\land}$\textcolor{green}{c}$^*$ & \textcolor{green}{d}$^*$ & 1$^*$ & 1 & 1 & 1 & 1 & $1_{\hspace{.3em}\land}$  & 1  & 1  & 2$^*$  \\  

\hline \\

B$_{12}$ & B$_{13}$ & B$_{14}$ & B$_{15}$ & B$_{16}$ & B$_{17}$ & B$_{18}$ & B$_{19}$ & B$_{20}$& B$_{21}$ & B$_{22}$ \\

6 & 6 & 6 & 6 & 4$^*$ & 4 & 4 & 4 & 3$^*$ & 3 & 3 \\

d  & d  & b$^*$  & a$^*$  & a & c$^*$ & 2$^*$ & 2 & 2 & 2 & \textcolor{green}{d}$^*$ \\

4$^*$  & 5$^*$  & 5  & 5  & $5_{\hspace{.3em}\land}$ & 5 & 5 & b$^*$ & b$_{\hspace{.3em}\land}$ & c$^*$ & \textcolor{green}{c}$_{\hspace{.3em}\land}$\\

\end{tabular}
\end{table}

\begin{theorem}[Build disjoint capable] \label{Theorem: create a disjoint-capable outer expansion set tsccd}
If there exists a \sccd$(v,k,b)$ with excess $e$ and a \sccd$(v',k,b')$ with excess $e'$ both with outer expansion sets that use both $U_0, U_b$ and $U'_0, U'_b$ respectively where $|X \cap X'|=k-1$, then for $v*=v+v'-k+1$, $b* = b + b' + \frac{(v'-k+1)(v-k+1)}{k-1}$, and $X^*= X \cup X'$ there exists a \sccd$(v^*,k,b^*)$ with excess $e^* = e + e'$ and a disjoint-capable outer expansion set.
\end{theorem}

\begin{proof}
Apply Theorem~\ref{Theorem: tsccd} to $(X,\cL)$ and $(X',\cL')$ to construct a  \tsccd$(v^*,k,b^*)$, $(X^*,\cL^*)$, with an outer expansion set using $U_0$ and $U_b'$. Insert the last $k-1$ blocks $B_i\cup \{x\}$ so that the last $k-1$ elements introduced are every $x\in U_0$. Reverse the resulting \sccd.
\end{proof}

For example, consider the \tsccd(21,4,69) in Table~\ref{Table: tsccd(21,4,2,69)}. The orange elements denote the unchanged subset and the green elements highlight the elements in $U'_b$ that are introduced in the first $k-1$ blocks. The block numbers that are highlighted in black, red and blue are the blocks of the \sccd(12,4,21), the \sccd(12,4,21), and the inserted $B''_{i_j,x}$ blocks respectively.

\begin{table}[H]
    \setlength{\tabcolsep}{.55\tabcolsep}
    \caption{A \tsccd(21,4,69) with a disjoint-capable outer expansion set. } 
    \label{Table: tsccd(21,4,2,69)} \label{Table: tsccd(21,4,69)}
    \label{Table: tsccd(12,4,2,21)} \label{Table: tsccd(12,4,21)} 
    \label{Table: tsccd(12,4,2,21) with double outer expansion set}
    \label{Table: tsccd(12,4,21) with double outer expansion set}
    \label{Table: tsccd(15,4,2,34)}     \label{Table: tsccd(15,4,34)}

    \begin{tabular}{p{14pt}p{10pt}p{10pt} p{10pt} p{10pt}p{10pt}p{10pt} p{10pt} p{10pt}p{10pt}p{10pt} p{10pt} p{10pt} p{10pt}p{10pt}p{10pt} p{10pt} p{10pt}p{10pt}p{10pt}    p{10pt}}
        $\hspace{0.6em}$\textcolor{blue}{\small{B}\tiny{$_{1}$}} & \textcolor{blue}{\small{B}\tiny{$_{2}$}} & \textcolor{blue}{\small{B}\tiny{$_{3}$}} & \textcolor{blue}{\small{B}\tiny{$_{4}$}} & \textcolor{blue}{\small{B}\tiny{$_{5}$}} & \textcolor{blue}{\small{B}\tiny{$_{6}$}} & \textcolor{blue}{\small{B}\tiny{$_{7}$}} & \textcolor{blue}{\small{B}\tiny{$_{8}$}} & \textcolor{blue}{\small{B}\tiny{$_{9}$}} & \small{B}\tiny{$_{10}$} & \small{B}\tiny{$_{11}$} & \small{B}\tiny{$_{12}$} & \small{B}\tiny{$_{13}$} & \small{B}\tiny{$_{14}$} & \small{B}\tiny{$_{15}$} & \textcolor{blue}{\small{B}\tiny{$_{16}$}} & \textcolor{blue}{\small{B}\tiny{$_{17}$}} & \textcolor{blue}{\small{B}\tiny{$_{18}$}} & \textcolor{blue}{\small{B}\tiny{$_{19}$}} & \textcolor{blue}{\small{B}\tiny{$_{20}$}} \\
        $\hspace{0.9em}\textcolor{orange}{1}^*$       & \textcolor{orange}{1}   & \textcolor{orange}{1}       & 1   & 1   & 1   & 1   & 1       & 1   & 1   & 1   & 1   & 1       & 1   & 1       & 13$^*$ & 14$^*$ & 15$^*$ & 16$^*$ & 17$^*$ \\
        
        $\hspace{0.9em}\textcolor{orange}{2}^*$       & \textcolor{orange}{2}   & \textcolor{orange}{2}       & 2   & 2   & 2   & 2   & 2       & 2   & 2   & 2   & 2   & 2       & 8$^*$  & 8       & 8   & 8   & 8   & 8   & 8   \\
        
        $\hspace{0.9em}\textcolor{orange}{3}^*$       & \textcolor{orange}{3}   & \textcolor{orange}{3}       & 3   & 3   & 3   & 3   & 3       & 3   & 3   & 3   & 3   & 7$^*$      & 7   & 7       & 7   & 7   & 7   & 7   & 7   \\
        
        $_{\hspace{.3em}\land}\textcolor{green}{13}^*$ & \textcolor{green}{14}$^*$ & \textcolor{green}{15}$^*$     & 16$^*$ & 17$^*$ & 18$^*$ & 19$^*$ & 20$^*$     & 21$^*$ & 4$^*$  & 5$^*$  & 6$^*$  & 6       & 6   & 9$^*$$_{\hspace{.3em}\land}$  & 9   & 9   & 9   & 9   & 9   \\ \hline
        
        \textcolor{blue}{\small{B}\tiny{$_{21}$}} & \textcolor{blue}{\small{B}\tiny{$_{22}$}} & \textcolor{blue}{\small{B}\tiny{$_{23}$}} & \textcolor{blue}{\small{B}\tiny{$_{24}$}} & \small{B}\tiny{$_{25}$} & \small{B}\tiny{$_{26}$} & \small{B}\tiny{$_{27}$} & \small{B}\tiny{$_{28}$} & \small{B}\tiny{$_{29}$} & \small{B}\tiny{$_{30}$} & \small{B}\tiny{$_{31}$} & \small{B}\tiny{$_{32}$} & \small{B}\tiny{$_{33}$} & \textcolor{blue}{\small{B}\tiny{$_{34}$}} & \textcolor{blue}{\small{B}\tiny{$_{35}$}} & \textcolor{blue}{\small{B}\tiny{$_{36}$}} & \textcolor{blue}{\small{B}\tiny{$_{37}$}} & \textcolor{blue}{\small{B}\tiny{$_{38}$}} & \textcolor{blue}{\small{B}\tiny{$_{39}$}} & \textcolor{blue}{\small{B}\tiny{$_{40}$}} \\ 
        
        18$^*$     & 19$^*$ & 20$^*$     & 21$^*$ & 3$^*$  & 10$^*$ & 10  & 11$^*$     & 11  & 11  & 11  & 11  & 11      & 11  & 11      & 11  & 11  & 11  & 11  & 11  \\
        
        8       & 8   & 8       & 8   & 8   & 8   & 8   & 8       & 8   & 8   & 8   & 9$^*$  & 10$^*$     & 10  & 10      & 10  & 10  & 10  & 10  & 10  \\
        
        7       & 7   & 7       & 7   & 7   & 7   & 7   & 7       & 7   & 7   & 2$^*$  & 2   & 2       & 13$^*$ & 14$^*$     & 15$^*$ & 16$^*$ & 17$^*$ & 18$^*$ & 19$^*$ \\
        
        9       & 9   & 9       & 9   & 9   & 9   & $5^*$  & 5       & 4$^*$  & 12$^*$ & 12  & 12  & $12_{\hspace{.3em}\land}$ & 12  & 12      & 12  & 12  & 12  & 12  & 12   \\ \hline

        \textcolor{blue}{\small{B}\tiny{$_{41}$}} & \textcolor{blue}{\small{B}\tiny{$_{42}$}} & \small{B}\tiny{$_{43}$} & \small{B}\tiny{$_{44}$} & \small{B}\tiny{$_{45}$} & \small{B}\tiny{$_{46}$} & \small{B}\tiny{$_{47}$} & \small{B}\tiny{$_{48}$} & \textcolor{red}{\small{B}\tiny{$_{49}$}} & \textcolor{red}{\small{B}\tiny{$_{50}$}} & \textcolor{red}{\small{B}\tiny{$_{51}$}} & \textcolor{red}{\small{B}\tiny{$_{52}$}} & \textcolor{red}{\small{B}\tiny{$_{53}$}} & \textcolor{red}{\small{B}\tiny{$_{54}$}} & \textcolor{red}{\small{B}\tiny{$_{55}$}} & \textcolor{red}{\small{B}\tiny{$_{56}$}} & \textcolor{red}{\small{B}\tiny{$_{57}$}} & \textcolor{red}{\small{B}\tiny{$_{58}$}} & \textcolor{red}{\small{B}\tiny{$_{59}$}} & \textcolor{red}{\small{B}\tiny{$_{60}$}}  \\
        
        11      & 11  & 11      & 11  & 11  & 4$^*$  & 4   & 4       & 4   & 4   & 4   & 4   & 4       & 4   & 6$^*$      & 19$^*$ & 19  & 20$^*$ & 20  & 20  \\
        10      & 10  & 10      & 10  & 10  & 10  & 5$^*$  & 5       & 5   & 5   & 5   & 5   & 17$^*$     & 17  & 17      & 17  & 17  & 17  & 17  & 17  \\
        20$^*$     & 21$^*$ & 1$^*$      & 3$^*$  & 6$^*$  & 6   & 6   & 6       & 6   & 6   & 6   & 16$^*$ & 16      & 16  & 16      & 16  & 16  & 16  & 16  & 16  \\
        12      & 12  & 12      & 12  & 12  & 12  & 12  & $9^*_{\hspace{.3em}\land}$  & 13$^*$ & 14$^*$ & 15$^*$ & 15  & 15      & 18$^*$ & $18_{\hspace{.3em}\land}$ & 18  & 14$^*$ & 14  & 13$^*$ & 21$^*$  \\ \hline
        
        \textcolor{red}{\small{B}\tiny{$_{61}$}} & \textcolor{red}{\small{B}\tiny{$_{62}$}} & \textcolor{red}{\small{B}\tiny{$_{63}$}} & \textcolor{red}{\small{B}\tiny{$_{64}$}} & \textcolor{red}{\small{B}\tiny{$_{65}$}} & \textcolor{red}{\small{B}\tiny{$_{66}$}} & \textcolor{red}{\small{B}\tiny{$_{67}$}} & \textcolor{red}{\small{B}\tiny{$_{68}$}} & \textcolor{red}{\small{B}\tiny{$_{69}$}} \\
        20      & 20  & 20      & 20  & 20  & 20  & 13$^*$ & 13      & \textcolor{green}{13}   \\
        17      & 18$^*$ & 19$^*$     & 19  & 19  & 19  & 19  & 14$^*$     & \textcolor{green}{14}   \\
        5$^*$      & 5   & 5       & 4$^*$  & 6$^*$  & 15$^*$ & 15  & 15      & \textcolor{green}{15}   \\
        21      & 21  & $21_{\hspace{.3em}\land}$ & 21  & 21  & 21  & 21  & 21  & $18^*_{\hspace{.3em}\land}$    
    \end{tabular}
\end{table}

\begin{corollary}\label{Corollary: tsccd(36,5,3,156) with disjoint capable expansion set exists}
    There exists a \tsccd(36,5,3,156) with a disjoint capable expansion set. 
\end{corollary}

\begin{proof}
    One of the \tsccd(20,5,46) from \cite{Phillips_20} has an expansion set containing both $U_0$ and $U_b$. Apply Theorem~\ref{Theorem: create a disjoint-capable outer expansion set tsccd} to two copies.
\end{proof}

This \sccd(36,5,156) is shown in Table~\ref{Table: tsccd(36,5,2,156) disjoint capable}. The black, red, and blue blocks are two copies of a \tsccd(20,5,46) and the inserted $B''_{i_j,x}$ blocks respectively.

\begin{table}[H]
\setlength{\tabcolsep}{.55\tabcolsep}
\centering
\caption{A \tsccd(36,5,156) with a disjoint-capable outer expansion set} 
\label{Table: tsccd(36,5,2,156) disjoint capable}
\label{Table: tsccd(20,5,2,46) - origional}
\label{tsccd(20,5)'} 
\label{Table: tsccd(20,5,2,46) Relabeled 1}
\label{Table: tsccd(36,5,156) disjoint capable}
\label{Table: tsccd(20,5,46) - origional}
 \label{Table: tsccd(20,5,46) Relabeled 1}
\footnotesize
\begin{tabular}{p{14pt}p{10pt}p{10pt} p{10pt} p{10pt}p{10pt}p{10pt} p{10pt} p{10pt}p{10pt}p{10pt} p{10pt} p{10pt} p{10pt}p{10pt}p{10pt} p{10pt} p{10pt}p{10pt}p{10pt}}

$\hspace{0.6em}$\small{B}\tiny{$_{1}$} & \small{B}\tiny{$_{2}$} & \small{B}\tiny{$_{3}$} & \small{B}\tiny{$_{4}$} & \small{B}\tiny{$_{5}$} & \small{B}\tiny{$_{6}$} & \small{B}\tiny{$_{7}$} & \small{B}\tiny{$_{8}$} & \small{B}\tiny{$_{9}$} & \small{B}\tiny{$_{10}$} & \small{B}\tiny{$_{11}$} & \small{B}\tiny{$_{12}$} & \small{B}\tiny{$_{13}$} & \small{B}\tiny{$_{14}$} & \small{B}\tiny{$_{15}$} & \small{B}\tiny{$_{16}$} & \small{B}\tiny{$_{17}$} & \small{B}\tiny{$_{18}$} & \small{B}\tiny{$_{19}$} & \small{B}\tiny{$_{20}$} \\
$\hspace{0.6em}$\textcolor{blue}{25}$^*$&\textcolor{blue}{29}$^*$&\textcolor{blue}{30}$^*$&\textcolor{blue}{34}$^*$&\textcolor{blue}{21}$^*$&\textcolor{blue}{22}$^*$&\textcolor{blue}{23}$^*$&\textcolor{blue}{24}$^*$&\textcolor{blue}{26}$^*$&\textcolor{blue}{27}$^*$&\textcolor{blue}{28}$^*$&\textcolor{blue}{31}$^*$&\textcolor{blue}{32}$^*$&\textcolor{blue}{33}$^*$&\textcolor{blue}{35}$^*$&\textcolor{blue}{36}$^*$&  1$^*$ &  1      & 1     & 1           \\

$\hspace{0.6em}$\textcolor{blue}{2}$^*$&\textcolor{blue}{2}&\textcolor{blue}{2}&\textcolor{blue}{2}&\textcolor{blue}{2}&\textcolor{blue}{2}&\textcolor{blue}{2}&\textcolor{blue}{2}&\textcolor{blue}{2}&\textcolor{blue}{2}&\textcolor{blue}{2}&\textcolor{blue}{2}&\textcolor{blue}{2}&\textcolor{blue}{2}&\textcolor{blue}{2}&\textcolor{blue}{2}        & 2      & 2      & 2     & 2      \\

$\hspace{0.6em}$\textcolor{blue}{3}$^*$ &\textcolor{blue}{3}&\textcolor{blue}{3}&\textcolor{blue}{3}&\textcolor{blue}{3}&\textcolor{blue}{3}&\textcolor{blue}{3}&\textcolor{blue}{3}&\textcolor{blue}{3}&\textcolor{blue}{3}&\textcolor{blue}{3}&\textcolor{blue}{3}&\textcolor{blue}{3}&\textcolor{blue}{3}&\textcolor{blue}{3}&\textcolor{blue}{3}       & 3      & 3      & 3     & 8$^*$  \\

$\hspace{0.6em}$\textcolor{blue}{4}$^*$ &\textcolor{blue}{4}&\textcolor{blue}{4}&\textcolor{blue}{4}&\textcolor{blue}{4}&\textcolor{blue}{4}&\textcolor{blue}{4}&\textcolor{blue}{4}&\textcolor{blue}{4}&\textcolor{blue}{4}&\textcolor{blue}{4}&\textcolor{blue}{4}&\textcolor{blue}{4}&\textcolor{blue}{4}&\textcolor{blue}{4}&\textcolor{blue}{4}& 4      & 4      & 7$^*$ & 7     \\

$_{\land}$\textcolor{blue}{5}$^*$  &\textcolor{blue}{5}&\textcolor{blue}{5}&\textcolor{blue}{5}&\textcolor{blue}{5}&\textcolor{blue}{5}&\textcolor{blue}{5}&\textcolor{blue}{5}&\textcolor{blue}{5}&\textcolor{blue}{5}&\textcolor{blue}{5}&\textcolor{blue}{5}&\textcolor{blue}{5}&\textcolor{blue}{5}&\textcolor{blue}{5}&\textcolor{blue}{5}     & 5     & 6$^*$  & 6     & 6     \\ \hline

\small{B}\tiny{$_{21}$} & \small{B}\tiny{$_{22}$} & \small{B}\tiny{$_{23}$} & \small{B}\tiny{$_{24}$} & \small{B}\tiny{$_{25}$} & \small{B}\tiny{$_{26}$} & \small{B}\tiny{$_{27}$} & \small{B}\tiny{$_{28}$} & \small{B}\tiny{$_{29}$} & \small{B}\tiny{$_{30}$} & \small{B}\tiny{$_{31}$} & \small{B}\tiny{$_{32}$} & \small{B}\tiny{$_{33}$} & \small{B}\tiny{$_{34}$} & \small{B}\tiny{$_{35}$} & \small{B}\tiny{$_{36}$} & \small{B}\tiny{$_{37}$} & \small{B}\tiny{$_{38}$} & \small{B}\tiny{$_{39}$} & \small{B}\tiny{$_{40}$} \\
 1      & 1   &\textcolor{blue}{1}&\textcolor{blue}{1}&\textcolor{blue}{1}&\textcolor{blue}{1}&\textcolor{blue}{1}&\textcolor{blue}{1}&\textcolor{blue}{1}&\textcolor{blue}{1}&\textcolor{blue}{1}&\textcolor{blue}{1}&\textcolor{blue}{1}&\textcolor{blue}{1}&\textcolor{blue}{1}&\textcolor{blue}{1}&\textcolor{blue}{1}&\textcolor{blue}{1}   & 1      & 1     \\

 9$^*$  & 9    &\textcolor{blue}{21}$^*$&\textcolor{blue}{22}$^*$&\textcolor{blue}{23}$^*$&\textcolor{blue}{24}$^*$&\textcolor{blue}{25}$^*$&\textcolor{blue}{26}$^*$&\textcolor{blue}{27}$^*$&\textcolor{blue}{28}$^*$&\textcolor{blue}{29}$^*$&\textcolor{blue}{30}$^*$&\textcolor{blue}{31}$^*$&\textcolor{blue}{32}$^*$&\textcolor{blue}{33}$^*$&\textcolor{blue}{34}$^*$&\textcolor{blue}{35}$^*$&\textcolor{blue}{36}$^*$            &{11}$^*$&       12$^*$ \\

8      & 10$^*$   &\textcolor{blue}{10}&\textcolor{blue}{10}&\textcolor{blue}{10}&\textcolor{blue}{10}&\textcolor{blue}{10}&\textcolor{blue}{10}&\textcolor{blue}{10}&\textcolor{blue}{10}&\textcolor{blue}{10}&\textcolor{blue}{10}&\textcolor{blue}{10}&\textcolor{blue}{10}&\textcolor{blue}{10}&\textcolor{blue}{10}&\textcolor{blue}{10}&\textcolor{blue}{10}    & 10     & 10     \\

7      & 7      &\textcolor{blue}{7}&\textcolor{blue}{7}&\textcolor{blue}{7}&\textcolor{blue}{7}&\textcolor{blue}{7}&\textcolor{blue}{7}&\textcolor{blue}{7}&\textcolor{blue}{7}&\textcolor{blue}{7}&\textcolor{blue}{7}&\textcolor{blue}{7}&\textcolor{blue}{7}&\textcolor{blue}{7}&\textcolor{blue}{7}&\textcolor{blue}{7}&\textcolor{blue}{7}      & 7      & 7      \\

6      & 6$_{\land}$  &\textcolor{blue}{6}&\textcolor{blue}{6}&\textcolor{blue}{6}&\textcolor{blue}{6}&\textcolor{blue}{6}&\textcolor{blue}{6}&\textcolor{blue}{6}&\textcolor{blue}{6}&\textcolor{blue}{6}&\textcolor{blue}{6}&\textcolor{blue}{6}&\textcolor{blue}{6}&\textcolor{blue}{6}&\textcolor{blue}{6}&\textcolor{blue}{6}&\textcolor{blue}{6}      & 6      & 6      \\ \hline

\small{B}\tiny{$_{41}$} & \small{B}\tiny{$_{42}$} & \small{B}\tiny{$_{43}$} & \small{B}\tiny{$_{44}$} & \small{B}\tiny{$_{45}$} & \small{B}\tiny{$_{46}$} & \small{B}\tiny{$_{47}$} & \small{B}\tiny{$_{48}$} & \small{B}\tiny{$_{49}$} & \small{B}\tiny{$_{50}$} & \small{B}\tiny{$_{51}$} & \small{B}\tiny{$_{52}$} & \small{B}\tiny{$_{53}$} & \small{B}\tiny{$_{54}$} & \small{B}\tiny{$_{55}$} & \small{B}\tiny{$_{56}$} & \small{B}\tiny{$_{57}$} & \small{B}\tiny{$_{58}$} & \small{B}\tiny{$_{59}$} & \small{B}\tiny{$_{60}$} \\
 1 &1      & 1      & 1      & 1           & 1      & 1      & 1      & 5$^*$ & 5     & 5      & 2$^*$  & 11$^*$        &\textcolor{blue}{11}&\textcolor{blue}{11}&\textcolor{blue}{11}&\textcolor{blue}{11}&\textcolor{blue}{11}&\textcolor{blue}{11}&\textcolor{blue}{11}\\

12     & 12     & 12     & 12     & 12          & 12     & 12     & 12     & 12    & 12    & 12     & 12   &12        &\textcolor{blue}{12}&\textcolor{blue}{12}&\textcolor{blue}{12}&\textcolor{blue}{12}&\textcolor{blue}{12}&\textcolor{blue}{12}&\textcolor{blue}{12}  \\

13$^*$ & 14$^*$ & 14     & 14     & 14          & 18$^*$ & 19$^*$ & 19     & 19    & 19    & 19     & 19   &19       &\textcolor{blue}{19}&\textcolor{blue}{19}&\textcolor{blue}{19}&\textcolor{blue}{19}&\textcolor{blue}{19}&\textcolor{blue}{19}&\textcolor{blue}{19}  \\

 7      & 7      & 15$^*$ & 16$^*$ & 17$^*$      & 17     & 17     & 20$^*$ & 20    & 20    & 20     & 20 &20        &\textcolor{blue}{20}&\textcolor{blue}{20}&\textcolor{blue}{20}&\textcolor{blue}{20}&\textcolor{blue}{20}&\textcolor{blue}{20}&\textcolor{blue}{20}\\

6      & 6      & 6      & 6      & 6           & 6      & 6      & 6      & 6     & 8$^*$ & 9$^*$  & 9   &9$_{\land}$       &\textcolor{blue}{21}$^*$&\textcolor{blue}{22}$^*$&\textcolor{blue}{23}$^*$&\textcolor{blue}{24}$^*$&\textcolor{blue}{25}$^*$&\textcolor{blue}{26}$^*$&\textcolor{blue}{27}$^*$   \\ \hline

\small{B}\tiny{$_{61}$} & \small{B}\tiny{$_{62}$} & \small{B}\tiny{$_{63}$} & \small{B}\tiny{$_{64}$} & \small{B}\tiny{$_{65}$} & \small{B}\tiny{$_{66}$} & \small{B}\tiny{$_{67}$} & \small{B}\tiny{$_{68}$} & \small{B}\tiny{$_{69}$} & \small{B}\tiny{$_{70}$} & \small{B}\tiny{$_{71}$} & \small{B}\tiny{$_{72}$} & \small{B}\tiny{$_{73}$} & \small{B}\tiny{$_{74}$} & \small{B}\tiny{$_{75}$} & \small{B}\tiny{$_{76}$} & \small{B}\tiny{$_{77}$} & \small{B}\tiny{$_{78}$} & \small{B}\tiny{$_{79}$} & \small{B}\tiny{$_{80}$} \\
$\hspace{0.4em}$\textcolor{blue}{11}&\textcolor{blue}{11}&\textcolor{blue}{11}&\textcolor{blue}{11}&\textcolor{blue}{11}&\textcolor{blue}{11}&\textcolor{blue}{11}&\textcolor{blue}{11}&\textcolor{blue}{11}        & 11     & 11    & 11     & 11     & 11           & 7$^*$  & 15$^*$ & 15     & 15     & 15     & 15      \\

$\hspace{0.4em}$\textcolor{blue}{12}&\textcolor{blue}{12}&\textcolor{blue}{12}&\textcolor{blue}{12}&\textcolor{blue}{12}&\textcolor{blue}{12}&\textcolor{blue}{12}&\textcolor{blue}{12}&\textcolor{blue}{12}            & 12     & 12    & 13$^*$ & 14$^*$ & 18$^*$       & 18     & 18     & 10$^*$ & 16$^*$ & 16     & 16      \\

$\hspace{0.4em}$\textcolor{blue}{19}&\textcolor{blue}{19}&\textcolor{blue}{19}&\textcolor{blue}{19}&\textcolor{blue}{19}&\textcolor{blue}{19}&\textcolor{blue}{19}&\textcolor{blue}{19}&\textcolor{blue}{19}             & 19     & 19    & 19     & 19     & 19           & 19     & 19     & 19     & 19     & 17$^*$ & 17      \\

$\hspace{0.4em}$\textcolor{blue}{20}&\textcolor{blue}{20}&\textcolor{blue}{20}&\textcolor{blue}{20}&\textcolor{blue}{20}&\textcolor{blue}{20}&\textcolor{blue}{20}&\textcolor{blue}{20}&\textcolor{blue}{20}            & 20     & 20    & 20     & 20     & 20           & 20     & 20     & 20     & 20     & 20     & 9$^*$    \\

$\hspace{0.4em}$\textcolor{blue}{28}$^*$&\textcolor{blue}{29}$^*$&\textcolor{blue}{30}$^*$&\textcolor{blue}{31}$^*$&\textcolor{blue}{32}$^*$&\textcolor{blue}{33}$^*$&\textcolor{blue}{34}$^*$&\textcolor{blue}{35}$^*$&\textcolor{blue}{36}$^*$       & 3$^*$  & 4$^*$ & 4      & 4      & 4            & 4      & 4      & 4      & 4      & 4      & 4       \\ \hline 

\small{B}\tiny{$_{81}$} & \small{B}\tiny{$_{82}$} & \small{B}\tiny{$_{83}$} & \small{B}\tiny{$_{84}$} & \small{B}\tiny{$_{85}$} & \small{B}\tiny{$_{86}$} & \small{B}\tiny{$_{87}$} & \small{B}\tiny{$_{88}$} & \small{B}\tiny{$_{89}$} & \small{B}\tiny{$_{90}$} & \small{B}\tiny{$_{91}$} & \small{B}\tiny{$_{92}$} & \small{B}\tiny{$_{93}$} & \small{B}\tiny{$_{94}$} & \small{B}\tiny{$_{95}$} & \small{B}\tiny{$_{96}$} & \small{B}\tiny{$_{97}$} & \small{B}\tiny{$_{98}$} & \small{B}\tiny{$_{99}$} & \small{B}\tiny{$_{100}$} \\
$\hspace{0.4em}$15    &\textcolor{blue}{15}&\textcolor{blue}{15}&\textcolor{blue}{15}&\textcolor{blue}{15}&\textcolor{blue}{15}&\textcolor{blue}{15}&\textcolor{blue}{15}&\textcolor{blue}{15}&\textcolor{blue}{15}&\textcolor{blue}{15}&\textcolor{blue}{15}&\textcolor{blue}{15}&\textcolor{blue}{15}&\textcolor{blue}{15}&\textcolor{blue}{15}&\textcolor{blue}{15}      & 15     & 15     & 15     \\

$\hspace{0.4em}$16    &\textcolor{blue}{16}&\textcolor{blue}{16}&\textcolor{blue}{16}&\textcolor{blue}{16}&\textcolor{blue}{16}&\textcolor{blue}{16}&\textcolor{blue}{16}&\textcolor{blue}{16}&\textcolor{blue}{16}&\textcolor{blue}{16}&\textcolor{blue}{16}&\textcolor{blue}{16}&\textcolor{blue}{16}&\textcolor{blue}{16}&\textcolor{blue}{16}&\textcolor{blue}{16}      & 16     & 16     & 16    \\

$\hspace{0.4em}$17     &\textcolor{blue}{17}&\textcolor{blue}{17}&\textcolor{blue}{17}&\textcolor{blue}{17}&\textcolor{blue}{17}&\textcolor{blue}{17}&\textcolor{blue}{17}&\textcolor{blue}{17}&\textcolor{blue}{17}&\textcolor{blue}{17}&\textcolor{blue}{17}&\textcolor{blue}{17}&\textcolor{blue}{17}&\textcolor{blue}{17}&\textcolor{blue}{17}&\textcolor{blue}{17}     & 17     & 17     & 17    \\

$\hspace{0.4em}$8$^*$    &\textcolor{blue}{8}&\textcolor{blue}{8}&\textcolor{blue}{8}&\textcolor{blue}{8}&\textcolor{blue}{8}&\textcolor{blue}{8}&\textcolor{blue}{8}&\textcolor{blue}{8}&\textcolor{blue}{8}&\textcolor{blue}{8}&\textcolor{blue}{8}&\textcolor{blue}{8}&\textcolor{blue}{8}&\textcolor{blue}{8}&\textcolor{blue}{8}&\textcolor{blue}{8}   & 8      & 8      & 2$^*$  \\

$\hspace{0.4em}$4$_{\land}$     &\textcolor{blue}{21}$^*$&\textcolor{blue}{22}$^*$&\textcolor{blue}{23}$^*$&\textcolor{blue}{24}$^*$&\textcolor{blue}{25}$^*$&\textcolor{blue}{26}$^*$&\textcolor{blue}{27}$^*$&\textcolor{blue}{28}$^*$&\textcolor{blue}{29}$^*$&\textcolor{blue}{30}$^*$&\textcolor{blue}{31}$^*$&\textcolor{blue}{32}$^*$&\textcolor{blue}{33}$^*$&\textcolor{blue}{34}$^*$&\textcolor{blue}{35}$^*$&\textcolor{blue}{36}$^*$        & 3$^*$  & 11$^*$ & 11     \\ \hline

\small{B}\tiny{$_{101}$} & \small{B}\tiny{$_{102}$} & \small{B}\tiny{$_{103}$} & \small{B}\tiny{$_{104}$} & \small{B}\tiny{$_{105}$} & \small{B}\tiny{$_{106}$} & \small{B}\tiny{$_{107}$} & \small{B}\tiny{$_{108}$} & \small{B}\tiny{$_{109}$} & \small{B}\tiny{$_{110}$} & \small{B}\tiny{$_{111}$} & \small{B}\tiny{$_{112}$} & \small{B}\tiny{$_{113}$} & \small{B}\tiny{$_{114}$} & \small{B}\tiny{$_{115}$} & \small{B}\tiny{$_{116}$} & \small{B}\tiny{$_{117}$} & \small{B}\tiny{$_{118}$} & \small{B}\tiny{$_{119}$} & \small{B}\tiny{$_{120}$} \\
$\hspace{0.4em}$15    & 15    & 15     & 10$^*$ & 10                  & 10     & 10    & 10     & 10     & 9$^*$        & \textcolor{red}{9} & \textcolor{red}{9}      & \textcolor{red}{23}$^*$ & \textcolor{red}{23}      & \textcolor{red}{23}      & \textcolor{red}{23}            & \textcolor{red}{23}      & \textcolor{red}{23}      & \textcolor{red}{23}      & \textcolor{red}{23}       \\

$\hspace{0.4em}$16    & 16    & 16     & 16 & 16                    & 14$^*$ & 14    & 14     & 14     & 14      &    \textcolor{red}{14} & \textcolor{red}{14}      & \textcolor{red}{14}     & \textcolor{red}{14}      & \textcolor{red}{25}$^*$  & \textcolor{red}{25}            & \textcolor{red}{27}$^*$ & \textcolor{red}{28}$^*$ & \textcolor{red}{28}     & \textcolor{red}{28}        \\

$\hspace{0.4em}$17    & 17    & 17     & 17  & 18$^*$                & 18     & 18    & 18     & 18     & 18            & \textcolor{red}{18} & \textcolor{red}{18}      & \textcolor{red}{18}     & \textcolor{red}{24}$^*$  & \textcolor{red}{24}      & \textcolor{red}{26}$^*$       & \textcolor{red}{26}     & \textcolor{red}{26}     & \textcolor{red}{29}$^*$ & \textcolor{red}{30}$^*$     \\

$\hspace{0.4em}$5$^*$ & 5     & 5      & 5 & 5                     & 5      & 2$^*$ & 8$^*$  & 3$^*$  & 3              & \textcolor{red}{21}$^*$ & \textcolor{red}{21}      & \textcolor{red}{21}     & \textcolor{red}{21}      & \textcolor{red}{21}      & \textcolor{red}{21}            & \textcolor{red}{21}      & \textcolor{red}{21}      & \textcolor{red}{21}      & \textcolor{red}{21}      \\

$\hspace{0.4em}$11    & 7$^*$ & 13$^*$ & 13  & 13                    & 13     & 13    & 13     & 13     & 13$_{\land}$  & \textcolor{red}{13}       & \textcolor{red}{22}$^*$  & \textcolor{red}{22}     & \textcolor{red}{22}      & \textcolor{red}{22}      & \textcolor{red}{22}$_{\land}$  & \textcolor{red}{22}      & \textcolor{red}{22}      & \textcolor{red}{22}      & \textcolor{red}{22}  \\ \hline

\small{B}\tiny{$_{121}$} & \small{B}\tiny{$_{122}$} & \small{B}\tiny{$_{123}$} & \small{B}\tiny{$_{124}$} & \small{B}\tiny{$_{125}$} & \small{B}\tiny{$_{126}$} & \small{B}\tiny{$_{127}$} & \small{B}\tiny{$_{128}$} & \small{B}\tiny{$_{129}$} & \small{B}\tiny{$_{130}$} & \small{B}\tiny{$_{131}$} & \small{B}\tiny{$_{132}$} & \small{B}\tiny{$_{133}$} & \small{B}\tiny{$_{134}$} & \small{B}\tiny{$_{135}$} & \small{B}\tiny{$_{136}$} & \small{B}\tiny{$_{137}$} & \small{B}\tiny{$_{138}$} & \small{B}\tiny{$_{139}$} & \small{B}\tiny{$_{140}$} \\
$\hspace{0.4em}$\textcolor{red}{31}$^*$ & \textcolor{red}{32}$^*$ & \textcolor{red}{33}$^*$      & \textcolor{red}{33}     & \textcolor{red}{33}     & \textcolor{red}{36}$^*$ & \textcolor{red}{36}    & \textcolor{red}{36}    & \textcolor{red}{36}     & \textcolor{red}{36}    &  \textcolor{red}{36}                    & \textcolor{red}{36}     & \textcolor{red}{36}    & \textcolor{red}{36}     & \textcolor{red}{36}     & \textcolor{red}{36}           & \textcolor{red}{36}     & \textcolor{red}{36}     & \textcolor{red}{36}     & \textcolor{red}{36}           \\

$\hspace{0.4em}$\textcolor{red}{28}     & \textcolor{red}{28}     & \textcolor{red}{28}          & \textcolor{red}{28}     & \textcolor{red}{28}     & \textcolor{red}{28}     & \textcolor{red}{28}    & \textcolor{red}{28}    & \textcolor{red}{28}     & \textcolor{red}{28} & \textcolor{red}{28}                    & \textcolor{red}{28}     & \textcolor{red}{28}    & \textcolor{red}{29}$^*$ & \textcolor{red}{30}$^*$ & \textcolor{red}{34}$^*$       & \textcolor{red}{34}     & \textcolor{red}{34}     & \textcolor{red}{26}$^*$ & \textcolor{red}{32}$^*$     \\

$\hspace{0.4em}$\textcolor{red}{30}     & \textcolor{red}{30}     & \textcolor{red}{30}          & \textcolor{red}{34}$^*$ & \textcolor{red}{35}$^*$ & \textcolor{red}{35}     & \textcolor{red}{35}    & \textcolor{red}{35}    & \textcolor{red}{35}     & \textcolor{red}{35} & \textcolor{red}{35}                    & \textcolor{red}{35}     & \textcolor{red}{35}    & \textcolor{red}{35}     & \textcolor{red}{35}     & \textcolor{red}{35}           & \textcolor{red}{35}     & \textcolor{red}{35}     & \textcolor{red}{35}     & \textcolor{red}{35}        \\

$\hspace{0.4em}$\textcolor{red}{21}      & \textcolor{red}{21}      & \textcolor{red}{21}           & \textcolor{red}{21}      & \textcolor{red}{21}      & \textcolor{red}{21}      & \textcolor{red}{13}$^*$ & \textcolor{red}{13}     & \textcolor{red}{13}      & \textcolor{red}{14}$^*$ & \textcolor{red}{27}$^*$                & \textcolor{red}{27}     & \textcolor{red}{27}    & \textcolor{red}{27}     & \textcolor{red}{27}     & \textcolor{red}{27}           & \textcolor{red}{23}$^*$  & \textcolor{red}{31}$^*$ & \textcolor{red}{31}     & \textcolor{red}{31}     \\

$\hspace{0.4em}$\textcolor{red}{22}      & \textcolor{red}{22}      & \textcolor{red}{22}           & \textcolor{red}{22}      & \textcolor{red}{22}      & \textcolor{red}{22}      & \textcolor{red}{22}     & \textcolor{red}{24}$^*$ & \textcolor{red}{25}$^*$  & \textcolor{red}{25}&  \textcolor{red}{25}$_{\land}$           &  \textcolor{red}{18}$^*$  & \textcolor{red}{9}$^*$ & \textcolor{red}{9}      & \textcolor{red}{9}      & \textcolor{red}{9}            & \textcolor{red}{9}      & \textcolor{red}{9}      & \textcolor{red}{9}      & \textcolor{red}{9}     \\ \hline

\small{B}\tiny{$_{141}$} & \small{B}\tiny{$_{142}$} & \small{B}\tiny{$_{143}$} & \small{B}\tiny{$_{144}$} & \small{B}\tiny{$_{145}$} & \small{B}\tiny{$_{146}$} & \small{B}\tiny{$_{147}$} & \small{B}\tiny{$_{148}$} & \small{B}\tiny{$_{149}$} & \small{B}\tiny{$_{150}$} & \small{B}\tiny{$_{151}$} & \small{B}\tiny{$_{152}$} & \small{B}\tiny{$_{153}$} & \small{B}\tiny{$_{154}$} & \small{B}\tiny{$_{155}$} & \small{B}\tiny{$_{156}$}  \\
$\hspace{0.4em}$\textcolor{red}{36}     & \textcolor{red}{25}$^*$  & \textcolor{red}{24}$^*$       & \textcolor{red}{24}      & \textcolor{red}{24}      & \textcolor{red}{14}$^*$  & \textcolor{red}{13}$^*$ & \textcolor{red}{13}     & \textcolor{red}{13}      & \textcolor{red}{13} & \textcolor{red}{13}                     & \textcolor{red}{13}      & \textcolor{red}{14}$^*$ & \textcolor{red}{24}$^*$  & \textcolor{red}{18}$^*$  & \textcolor{red}{18}           \\

$\hspace{0.4em}$\textcolor{red}{32}     & \textcolor{red}{32}     & \textcolor{red}{32}          & \textcolor{red}{32}     & \textcolor{red}{32}     & \textcolor{red}{32}     & \textcolor{red}{32}    & \textcolor{red}{32}    & \textcolor{red}{32}     & \textcolor{red}{32} & \textcolor{red}{32}                    & \textcolor{red}{30}$^*$ & \textcolor{red}{30}    & \textcolor{red}{30}     & \textcolor{red}{30}     & \textcolor{red}{30}          \\

$\hspace{0.4em}$\textcolor{red}{33}$^*$ & \textcolor{red}{33}     & \textcolor{red}{33}          & \textcolor{red}{33}     & \textcolor{red}{33}     & \textcolor{red}{33}     & \textcolor{red}{33}    & \textcolor{red}{33}    & \textcolor{red}{33}     & \textcolor{red}{33} & \textcolor{red}{34}$^*$                & \textcolor{red}{34}     & \textcolor{red}{34}    & \textcolor{red}{34}     & \textcolor{red}{34}     & \textcolor{red}{34}          \\

$\hspace{0.4em}$\textcolor{red}{31}     & \textcolor{red}{31}     & \textcolor{red}{31}          & \textcolor{red}{31}     & \textcolor{red}{31}     & \textcolor{red}{31}     & \textcolor{red}{31}    & \textcolor{red}{31}    & \textcolor{red}{31}     & \textcolor{red}{26}$^*$ & \textcolor{red}{26}                  & \textcolor{red}{26}     & \textcolor{red}{26}    & \textcolor{red}{26}     & \textcolor{red}{26}     & \textcolor{red}{25}$^*$        \\

$\hspace{0.4em}$\textcolor{red}{9}      & \textcolor{red}{9}      & \textcolor{red}{9}$_{\land}$ & \textcolor{red}{18}$^*$  & \textcolor{red}{27}$^*$ & \textcolor{red}{27}     & \textcolor{red}{27}    & \textcolor{red}{23}$^*$ & \textcolor{red}{29}$^*$ & \textcolor{red}{29} & \textcolor{red}{29}                    & \textcolor{red}{29}     & \textcolor{red}{29}    & \textcolor{red}{29}     & \textcolor{red}{29}     & \textcolor{red}{29}$_{\land}$    
\end{tabular}
\end{table}

We bring the previous lemmas and theorems together to construct circular \sccd s.

\begin{theorem}[Constructing \scccd ~from two \sccd] \label{Theorem: recursive construction for tscccd(v,k,2)}\label{Theorem: 2 sccd to build a circular sccd or scccd} \label{Theorem: sccd + sccd = circular sccd}
If there exists a \sccd$(v,k,b)$ with excess $e$ and a disjoint-capable expansion set and there exists a \sccd$(v',k,b')$ with excess $e'$ where $|X \cap X'|=2k-2$, then for $v^* = v+v'-2k+2$ and $b^*= b+b'-k+1+ \frac{(v-2k+2)(v'-2k+2)}{k-1}$ a circular \sccd$(v^*,k,b^*)$ exists with  excess $e^*=e+e'$ and $X^*=X \cup X'$. Furthermore, if the \sccd$(v',k,b')$ has an outer expansion set using both $U'_0$ and $U'_{b'}$ then the circular \sccd$(v+v'-2k+2,k,b^*)$ has an expansion set. 
\end{theorem}
\begin{proof}
Suppose that $(X,\mathcal{L})$ is a $\sccd(v,k,b)$ with a disjoint capable outer expansion set $\mathcal{E}= \{U_{i_j}: 1 \le j \le \frac{v}{k-1}\}$. Suppose that $(X',\mathcal{L}')$ is a $\sccd(v',k,b')$ with excess $e'$ and $X'$ relabeled so $X'\cap X = U_0 \cup U_b$ and $U_b \subseteq B_1$, $U_{0} \subseteq B_{b'}$. To build $(X^*,\mathcal{M}^*)$, the circular $\sccd(v+v'-2k+2,k,b^*)$, we delete the first $k-1$ blocks from $\mathcal{L}$ and append the blocks from $\mathcal{L}'$. Note, $\cE\backslash \{U_0, U_b\}$  partitions $X \backslash X'$. For  $2 \le j \le \frac{v}{k-1}-1$ and each $x\in X' \backslash X$ we construct $B''_{i_{j},x} = (B_{i_j} \cap B_{i_j+1}) \cup \{x\}$ and insert $B''_{i_{j},x}$ between $B_i$ and $B_{i+1}$ in any order. 

The pairs in $\{x,y\}$ from $X\backslash X'$ are covered only in $\cL$ so are covered in $\cM^*$. Similarly, the pairs $\{x,y\}$ from $X'$ are covered in $\cL'$ inside $\mathcal{M}^*$ with the same excess. $B''_{i_j,x}$ covers $\{x,y\}$ for all $x \in X'\backslash X$ and $y \in X\backslash X'$ and no other $B''_{i_j,x}$ covers this pair. Moreover, $U_k = U_0 \subset B'_{b'}$ so there is a single change between the last block of the design and the first block. Therefore $(X^*,\mathcal{M}^*)$ is a circular $\sccd$.

\begin{align*}
    e^* = & (k-1)b^* - \binom{v^*}{2}\\
        = & (k-1)(b+b'-k+1+\frac{(v-2k+2)(v'-2k+2)}{k-1}) \\
          & -\binom{v +v' - 2k+2}{2} \\
        = & (k-1)b + (k-1)b' - (k-1)^2 +4(k-1)^2-\frac{v^2}{2}-\frac{v'^2}{2}-2(k-1)^2\\
          & +\frac{v}{2}+\frac{v'}{2}-(k-1)\\
        =& (k-1)b - \binom{v}{2} + \frac{(k-1)^2-(k-1)}{2} + (k-1)b' - \binom{v'}{2} + \frac{(k-1)^2-(k-1)}{2} \\
        =& (k-1)b - \binom{v}{2} + \binom{k-1}{2} + (k-1)b' - \binom{v'}{2} + \binom{k-1}{2} \\
        =& e + e'
\end{align*}


Suppose further that $(X',\mathcal{L}')$ has expansion set $\mathcal{E}' = \{U'_{i_j} : 1 \le j\le \frac{v'}{k-1}\}$. Then the circular \sccd, $(X^*, \mathcal{M}^*)$, will have an expansion set  $\mathcal{E}^* = \{(\mathcal{E} \cup \mathcal{E}') \backslash \{ U'_{i_0} , U'_{i_{b'}} \} \}.
$

\end{proof}

For example, consider a \sccd(6,3,7) as in Table~\ref{Table: tsccd(6,3,7) for use in disjoint capable expansion set circular construction} and the tight \sccd(10,3,22) in Table~\ref{Table: tsccd(10,3,2) extra property}, we combine these two \sccd s as prescribed in Theorem~\ref{Theorem: 2 sccd to build a circular sccd or scccd} to construct a \tscccd(12,3,33) as seen in Table~\ref{Table: tscccd(12,3,33) constructed via disjoint capable expansion sets}

\begin{table}[H]
\centering
\begin{tabular}{lllllll}
B$_1$ & B$_2$ & B$_3$ & B$_4$ & B$_5$ & B$_6$ & B$_7$ \\
11$^*$ & 11 & 11 & 11 & c$^*$  & c & d$^*$ \\
d$^*$  & d  & a$^*$  & b$^*$  & b  & b & b \\
$_{\land}$c$^*$  & 12$^*$ & 12$_{\land}$ & 12 & 12 & a$^*$ & a$_{\land}$
\end{tabular}
\caption{A \tsccd(6,3,7) Labeled for use in the Theorem~\ref{Theorem: 2 sccd to build a circular sccd or scccd} construction}
\label{Table: tsccd(6,3,7) for use in disjoint capable expansion set circular construction}
\end{table}

\begin{table}[H]
\centering
\begin{tabular}{lllllllllll}
B$_1$ & B$_2$ & B$_3$ & B$_4$ & B$_5$ & B$_6$ & B$_7$ & B$_8$ & B$_9$ & B$_{10}$ & B$_{11}$ \\
a  & a  & a & 4$^*$ & 5$^*$  & 6$^*$  & 6  & 6  & 6  & 6 & 6 \\
b  & 2$^*$  & 3$^*$ & 3 & 3  & 3  & 11$^*$ & 12$^*$ & c$^*$  & d$^*$ & d \\
1$^*$  & 1  & 1 & 1 & 1  & 1$_{\land}$  & 1  & 1  & 1  & 1 & 2$^*$ \\ \hline

B$_{12}$ & B$_{13}$ & B$_{14}$ & B$_{15}$ & B$_{16}$ & B$_{17}$ & B$_{18}$ & B$_{19}$ & B$_{20}$& B$_{21}$ & B$_{22}$ \\
6  & 6  & 6 & 6 & 4$^*$  & 4  & 4  & 4  & 4  & 4 & 3$^*$ \\
d  & d  & b$^*$ & a$^*$ & a  & 11$^*$ & 12$^*$ & c$^*$  & 2$^*$  & 2 & 2 \\
4$^*$  & 5$^*$  & 5 & 5 & 5$_{\land}$  & 5  & 5  & 5  & 5  & b$^*$ & b$_{\land}$ \\ \hline

B$_{23}$ & B$_{24}$ & B$_{25}$ & B$_{26}$ & B$_{27}$ & B$_{28}$ & B$_{29}$ & B$_{30}$& B$_{31}$ & B$_{32}$ & B$_{33}$ \\
3  & 3  & 3 & 3 & 11$^*$ & 11 & 11 & 11 & c$^*$  & c & d$^*$ \\
2  & 2  & 2 & d$^*$ & d  & d  & a$^*$  & b$^*$  & b  & b & b \\
11$^*$ & 12$^*$ & c$^*$ & c$_{\land}$ & c  & 12$^*$$_{\land}$ & 12 & 12 & 12 & a$^*$ & a$_{\land}$
\end{tabular}
\caption{A tight circular \sccd(12,3,33) constructed via Theorem~\ref{Theorem: 2 sccd to build a circular sccd or scccd}}
\label{Table: tscccd(12,3,33) constructed via disjoint capable expansion sets}
\end{table}

\subsection{Difference Methods}

We can also construct \scccd ~using difference methods.

\begin{theorem} [General Difference Construction for \scccd] \label{Theorem: family construction of circular scccd with fixed k}
\quad \newline There exists \tscccd$(2c(k-1)+1,k,c^2 (2 k - 2) + c)$ for all $c\geq 1$. 
\end{theorem}
\begin{proof}
Let $B_{i}=\{0,1,....k-2,(i+1)(k-1) \} \subseteq \mathbb{Z}_v$ for $0\leq i \leq c-1$. Let $L_j=(B_1+j, B_2+j,...,B_{c-1}+j,B_0+j)$ and $\cL$ be the concatenation of the $L_j$ in order $0\leq j\leq v$. $\cL$ is a circular single change list of blocks where the element introduced in block $B_i+j$, $1\leq i\leq c-1$ is $(i+1)(k-1)+j$. Thus $B_i$ covers the differences $\pm\{i+1,2(i+1),...,(k-1)(i+1)\}$. The element introduced in $B_0$ is $(i+1)$, thus $B_0$ covers the differences $\pm\{1,2,...,k\}$. This is all the differences in $\mathbb{Z}_v$ so as $\cL$ is formed by the development of $L_i$ all pairs are covered. The \scccd ~is tight as $b=g_2$ and is an integer.
\end{proof}

The \tscccd(13,4,26) where $c=2$ and \tscccd(19,4,57) where $c=3$ are given in 
Table~\ref{Table: tscccd(13,4,2,26)} and Table~\ref{Table: tscccd(19,4,2,57)} respectively.


\begin{table}[H]
\centering
\begin{tabular}{lllllllllllll}
    \small{B}\tiny{$_{1}$} & \small{B}\tiny{$_{2}$} & \small{B}\tiny{$_{3}$} & \small{B}\tiny{$_{4}$} & \small{B}\tiny{$_{5}$} & \small{B}\tiny{$_{6}$} & \small{B}\tiny{$_{7}$} & \small{B}\tiny{$_{8}$} & \small{B}\tiny{$_{9}$} & \small{B}\tiny{$_{10}$} & \small{B}\tiny{$_{11}$} & \small{B}\tiny{$_{12}$} & \small{B}\tiny{$_{13}$} \\
    0     & 0     & 7\ask  & 4\ask & 4      & 4     & 4      & 4     & 4      & 4     & 11\ask & 8\ask & 8      \\
    1     & 1     & 1      & 1     & 8\ask  & 5\ask & 5      & 5     & 5      & 5     & 5      & 5     & 12\ask \\
    2     & 2     & 2      & 2     & 2      & 2     & 9\ask  & 6\ask & 6      & 6     & 6      & 6     & 6      \\
    6\ask & 3\ask & 3      & 3     & 3      & 3     & 3      & 3     & 10\ask & 7\ask & 7      & 7     & 7      \\ \hline
    \small{B}\tiny{$_{14}$} & \small{B}\tiny{$_{15}$} & \small{B}\tiny{$_{16}$} & \small{B}\tiny{$_{17}$} & \small{B}\tiny{$_{28}$} & \small{B}\tiny{$_{19}$} & \small{B}\tiny{$_{20}$} & \small{B}\tiny{$_{21}$} & \small{B}\tiny{$_{22}$} & \small{B}\tiny{$_{23}$} & \small{B}\tiny{$_{24}$} & \small{B}\tiny{$_{25}$} & \small{B}\tiny{$_{26}$}\\
    8     & 8     & 8      & 8     & 8      & 2\ask & 12\ask & 12    & 12     & 12    & 12     & 12    & 12     \\
    9\ask & 9     & 9      & 9     & 9      & 9     & 9      & 3\ask & 0\ask  & 0     & 0      & 0     & 0      \\
    6     & 0\ask & 10\ask & 10    & 10     & 10    & 10     & 10    & 10     & 4\ask & 1\ask  & 1     & 1      \\
    7     & 7     & 7      & 1\ask & 11\ask & 11    & 11     & 11    & 11     & 11    & 11     & 5\ask & 2\ask 
\end{tabular}
\caption{A \tscccd(13,4,26)}
\label{Table: tscccd(13,4,2,26)}
\end{table}

\begin{table}[H]
\centering
\setlength{\tabcolsep}{.55\tabcolsep}
\begin{tabular}{llllllllllllllllllll}

\small{B}\tiny{$_{1}$} & \small{B}\tiny{$_{2}$} & \small{B}\tiny{$_{3}$} & \small{B}\tiny{$_{4}$} & \small{B}\tiny{$_{5}$} & \small{B}\tiny{$_{6}$} & \small{B}\tiny{$_{7}$} & \small{B}\tiny{$_{8}$} & \small{B}\tiny{$_{9}$} & \small{B}\tiny{$_{10}$} & \small{B}\tiny{$_{11}$} & \small{B}\tiny{$_{12}$} & \small{B}\tiny{$_{13}$} & \small{B}\tiny{$_{14}$} & \small{B}\tiny{$_{15}$} \\

0  & 0   & 0   & 7*  & 10* & 4*  & 4   & 4   & 4   & 4   & 4   & 4  & 4   & 4   & 4    \\

1  & 1   & 1   & 1   & 1   & 1   & 8*  & 11* & 5*  & 5   & 5   & 5  & 5   & 5   & 5   \\

2  & 2   & 2   & 2   & 2   & 2   & 2   & 2   & 2   & 9*  & 12* & 6* & 6   & 6   & 6     \\

6* & 9*  & 3*  & 3   & 3   & 3   & 3   & 3   & 3   & 3   & 3   & 3  & 10* & 13* & 7*  \\
\hline

\small{B}\tiny{$_{16}$} & \small{B}\tiny{$_{17}$} & \small{B}\tiny{$_{18}$} & \small{B}\tiny{$_{19}$} & \small{B}\tiny{$_{20}$} & \small{B}\tiny{$_{21}$} & \small{B}\tiny{$_{22}$} & \small{B}\tiny{$_{23}$} & \small{B}\tiny{$_{24}$} & \small{B}\tiny{$_{25}$} & \small{B}\tiny{$_{26}$} & \small{B}\tiny{$_{27}$} & \small{B}\tiny{$_{28}$} & \small{B}\tiny{$_{29}$} & \small{B}\tiny{$_{30}$} \\
11* & 14* & 8* & 8   & 8 & 8  & 8   & 8   & 8   & 8   & 8   & 8  & 15* & 18* & 12*  \\

5   & 5   & 5  & 12* & 15* & 9* & 9   & 9   & 9   & 9   & 9   & 9  & 9   & 9   & 9     \\

6   & 6   & 6  & 6   & 6   & 6  & 13* & 16* & 10* & 10  & 10  & 10  & 10  & 10  & 10  \\

7   & 7   & 7  & 7   & 7   & 7  & 7   & 7   & 7   & 14* & 17* & 11* & 11  & 11  & 11  \\
\hline

\small{B}\tiny{$_{31}$} & \small{B}\tiny{$_{32}$} & \small{B}\tiny{$_{33}$} & \small{B}\tiny{$_{34}$} & \small{B}\tiny{$_{35}$} & \small{B}\tiny{$_{36}$} & \small{B}\tiny{$_{37}$} & \small{B}\tiny{$_{38}$} & \small{B}\tiny{$_{39}$} & \small{B}\tiny{$_{40}$} &\small{B}\tiny{$_{41}$} & \small{B}\tiny{$_{42}$} & \small{B}\tiny{$_{43}$} & \small{B}\tiny{$_{44}$} & \small{B}\tiny{$_{45}$} \\
12  & 12 & 12  & 12  & 12 & 12  & 12  & 12 & 12  & 0*   & 3* & 16*  & 16  & 16  & 16 \\

16* & 0* & 13* & 13  & 13 & 13  & 13  & 13 & 13  & 13   & 13 & 13   & 1*  & 4*  & 17* \\

10  & 10 & 10  & 17* & 1* & 14* & 14  & 14 & 14  & 14  & 14 & 14    & 14  & 14  & 14 \\

11  & 11 & 11  & 11  & 11 & 11  & 18* & 2* & 15* & 15  & 15 & 15    & 15  & 15  & 15  \\
\hline

\small{B}\tiny{$_{46}$} & \small{B}\tiny{$_{47}$} & \small{B}\tiny{$_{48}$} & \small{B}\tiny{$_{49}$} & \small{B}\tiny{$_{50}$} & \small{B}\tiny{$_{51}$} & \small{B}\tiny{$_{52}$} & \small{B}\tiny{$_{53}$} & \small{B}\tiny{$_{54}$} & \small{B}\tiny{$_{55}$} & \small{B}\tiny{$_{56}$} & \small{B}\tiny{$_{57}$}  \\

16  & 16  & 16  & 16  & 16  & 16  & 4* & 7*  & 1*  & 1  & 1   & 1   &    &     &     \\

17  & 17  & 17  & 17  & 17  & 17  & 17 & 17  & 17  & 5* & 8*  & 2*  &    &     &     \\

2*  & 5*  & 18* & 18  & 18  & 18  & 18 & 18  & 18  & 18 & 18  & 18  &    &     &     \\

15  & 15  & 15  & 3*  & 6*  & 0*  & 0  & 0   & 0   & 0  & 0   & 0   &    &     &    

\end{tabular}
\caption{A \tscccd(19,4,57)}
\label{Table: tscccd(19,4,2,57)}
\end{table}

\section{Conclusion}
Using all the lemmas and theorems given here except Theorem~\ref{Theorem: Previously known groups of sccd}
along with the designs given 
we can now build the following designs. 

\begin{corollary}
    \begin{enumerate}
        \item An economical \scccd($v,4,b$) exists for all $v \geq 27$. These are tight if and only if $v\equiv 0,1 \pmod{3}$.
        \item An economical $\sccd(v,5,b)$ exists for all $v\equiv 4,5,6\pmod{16}$, $ v \geq 20$. These are tight if $v\equiv 4,5 \pmod{16}$
        \item An economical $\scccd(v,5,b)$ exists for all $v\equiv 0,1,2 \pmod{16}$, $ v \geq 48$. These are tight if $v\equiv 0,1 \pmod{16}$.
        \item A \tscccd$(2c(k-1)+1,k,c^2 (2 k - 2) + c + 1)$ exists for all $c\geq 1$, $k>2$.
    \end{enumerate}
\end{corollary}

\begin{proof}


To prove (1) we use a \tsccd(12,4,21), a \tsccd(15,4,35), a \tsccd(18,4,51), and the disjoint capable \tsccd(21,4,69) which are provided in the black blocks of Table~\ref{Table: tsccd(12,4,2,21)}, the red blocks of Table~\ref{Table: tsccd(15,4,34)}, Table~\ref{Table: tsccd(18,4,2,50)} and Table~\ref{Table: tsccd(21,4,69)} respectively. These tables in conjunction with Theorem~\ref{Theorem: esccd + tsccd = esccd}, Theorem~\ref{Theorem: Existance tsccd(v,4,2)},  Theorem~\ref{Theorem: sccd + sccd = circular sccd}, Proposition~\ref{Prop: circular sccd(v+1,k)}, and Theorem~\ref{Theorem: sccd(v) -> sccd(v+2)} complete the proof of (1).

To prove (2) we use the \tsccd(20,5,46) from the black blocks of Table~\ref{Table: tsccd(20,5,2,46) - origional} in conjunction with Proposition~\ref{Prop: circular sccd(v+1,k)}, Theorem ~\ref{Theorem: esccd + tsccd = esccd}, and Theorem~\ref{Theorem: sccd(v) -> sccd(v+2)} .

To prove (3) we use the \tscccd(20,5,46)  from the black blocks of Table~\ref{Table: tsccd(20,5,2,46) - origional} and the \tsccd(36,5,156) from Table~\ref{Table: tsccd(36,5,2,156) disjoint capable} in conjunction with (2), Proposition~\ref{Prop: circular sccd(v+1,k)}, Theorem~\ref{Theorem: sccd(v) -> sccd(v+2)}, and Theorem~\ref{Theorem: 2 sccd to build a circular sccd or scccd}.

To prove (4) we use the construction given in Theorem~\ref{Theorem: family construction of circular scccd with fixed k}.

\end{proof}

\begin{table}[H]
    \setlength{\tabcolsep}{.55\tabcolsep}
    \caption{A \tsccd(18,4,50) with outer expansion set \cite{tsccd}. }
    \label{Table: tsccd(18,4,2,50)}
    \begin{tabular}{p{10pt}p{10pt}p{10pt}p{10pt}p{10pt}p{10pt}p{10pt}p{10pt}p{10pt}p{10pt}p{10pt}p{10pt}p{10pt}p{10pt}p{10pt}p{10pt}p{10pt}p{10pt}p{10pt}p{10pt}p{10pt}}
        \small{B}\tiny{$_{1}$} & \small{B}\tiny{$_{2}$} & \small{B}\tiny{$_{3}$} & \small{B}\tiny{$_{4}$} & \small{B}\tiny{$_{5}$} & \small{B}\tiny{$_{6}$} & \small{B}\tiny{$_{7}$} & \small{B}\tiny{$_{8}$} & \small{B}\tiny{$_{9}$} & \small{B}\tiny{$_{10}$} & \small{B}\tiny{$_{11}$} & \small{B}\tiny{$_{12}$} & \small{B}\tiny{$_{13}$} & \small{B}\tiny{$_{14}$} & \small{B}\tiny{$_{15}$} & \small{B}\tiny{$_{16}$} & \small{B}\tiny{$_{17}$} & \small{B}\tiny{$_{18}$} & \small{B}\tiny{$_{19}$} & \small{B}\tiny{$_{20}$} \\
        1\ask & 1 & 1 & 1 & 1 & 1 & 1 & 1 & 1 & 1 & 1 & 1 & 1 & 1 & 1 & 4\ask & 4 & 4 & 4 & 4 \\
        2\ask & 2 & 2 & 2 & 2 & 2 & 2 & 2 & 2 & 2 & 2 & 2 & 16\ask & 17\ask & 17 & 17 & 17 & 17 & 17 & 17\\
        3\ask & 3 & 6\ask & 7\ask & 8\ask & 9\ask & 10\ask & 11\ask & 12\ask & 13\ask & 14\ask & 15\ask & 15 & 15 & 18\ask & 18 & 18 & 18 & 18 & 18\\
        4\ask & 5\ask & 5\car &  5 &  5 &  5 &  5 &  5 &  5 &  5 &  5 &  5 &  5 &  5 &  5 &  5 & 6\ask\car & 7\ask & 8\ask & 9\ask  \\ \hline
        
        \small{B}\tiny{$_{21}$} & \small{B}\tiny{$_{22}$} & \small{B}\tiny{$_{23}$} & \small{B}\tiny{$_{24}$} & \small{B}\tiny{$_{25}$} & \small{B}\tiny{$_{26}$} & \small{B}\tiny{$_{27}$} & \small{B}\tiny{$_{28}$} & \small{B}\tiny{$_{29}$} & \small{B}\tiny{$_{30}$} & \small{B}\tiny{$_{31}$} & \small{B}\tiny{$_{32}$} & \small{B}\tiny{$_{33}$} & \small{B}\tiny{$_{34}$} & \small{B}\tiny{$_{35}$} & \small{B}\tiny{$_{36}$} & \small{B}\tiny{$_{37}$} & \small{B}\tiny{$_{38}$} & \small{B}\tiny{$_{39}$} & \small{B}\tiny{$_{40}$} \\ 
        4 & 2\ask & 10\ask & 11\ask & 12\ask & 13\ask & 13  & 13 & 13 & 13 & 13 & 13 & 13 & 13 & 13 & 13 & 13 & 3\ask & 3 & 16\ask \\
        17 & 17 & 17 & 17 & 17 & 17 & 17 & 17 & 15\ask & 15 & 15 & 15 & 15 & 15 & 15 & 15 & 15 & 15 & 7\ask & 7\\
        18 & 18 & 18 & 18 & 18 & 18 & 18 & 18 & 18 & 4\ask & 4 & 4 & 4 & 6\ask & 6 & 6 & 9\ask & 9 & 9 & 9 \\
        16\ask  & 16 & 16 & 16 & 16 & 16 & 3\ask & 14\ask & 14\car & 14 & 10\ask & 11\ask & 12\ask & 12 & 7\ask & 8\ask & 8 & 8 & 8 & 8\car \\ \hline

        \small{B}\tiny{$_{41}$} & \small{B}\tiny{$_{42}$} & \small{B}\tiny{$_{43}$} & \small{B}\tiny{$_{44}$} & \small{B}\tiny{$_{45}$} & \small{B}\tiny{$_{46}$} & \small{B}\tiny{$_{47}$} & \small{B}\tiny{$_{48}$} & \small{B}\tiny{$_{49}$} & \small{B}\tiny{$_{50}$} \\
        16 & 16 & 16 & 10\ask & 10 & 10 & 10 & 10 & 10 & 14\ask \\
        7 & 6\ask  & 6 & 6 & 6 & 12\ask & 12 & 12 & 12 & 12 \\
        9 & 9 & 3\ask & 3 & 3 & 3 & 7\ask & 9\ask & 8\ask & 8 \\
        14\ask & 14 & 14 & 14\car& 11\ask & 11 & 11 & 11 & 11 & 11\car 
    \end{tabular}
\end{table}

With McSorely's \tscccd(9,4,12) and \tscccd(10,4,15), we need to find \tscccd$(v,4,b)$ for $v = 12,13,15,16,18,19,21,22,24,25$ to show a \tscccd$(v,4,b)$ exists for every admissible $v\geq 9$. 

The existence of a \tsccd(28,5,94) with an expansion set would allow the use of Proposition~\ref{Prop: circular sccd(v+1,k)}, Theorem~\ref{Theorem: tsccd +tsccd = tsccd}, and Theorem~\ref{Corollary: create esccd(v+2,k)} to show an economical \sccd($v,5,b$) exists for $v \equiv 4,5,6 \pmod{8}$, $v\geq 20$. These are tight if $v\equiv 4,5 \pmod{8}$. 
Using these designs in conjunction with Theorem~\ref{Theorem: 2 sccd to build a circular sccd or scccd} and the disjoint-capable \sccd(36,5,156) given in Table~\ref{Table: tsccd(36,5,2,156) disjoint capable} would allow us to construct \scccd($v,5,b$) for $v \equiv 0,1,2 \pmod{8}, v\geq 48$. 

We require a \scccd($v,5,b$) for $v=24,25,26,32,33,34,40,41,42$ to construct all \scccd($v,5,b$), for $v\equiv 0,1,2 \pmod{8}, v\geq 24$.
Finding an \esccd$(v,5,b)$ for $v=23,24,25,26,27,30,31,32,33,34,35$ would allow us to construct an \sccd$(v,5,b)$ for every admissible $v$. 

Finally finding \esccd$(v,5,b)$ and \tsccd$(v,5,b)$ for $ v = 24,25,32,33,40,41$ with a disjoint-capable outer set, would show the existence of a \escccd$(v,5,b)$ for every admissible $v$.

\bibliographystyle{plain}
\bibliography{References}

\end{document}